\documentclass[12pt,reqno]{amsart}

\usepackage[T1]{fontenc}
\usepackage{enumitem}
\usepackage{hyperref}
\hypersetup{
  colorlinks   = true,
  citecolor    = blue,
  linkcolor    = blue
}

\usepackage{amsmath,amsfonts,amsthm,amssymb,color,tikz, comment}
\usepackage[english]{babel}
\usepackage{mathrsfs}
\usepackage{comment}

\usepackage{pdfsync}
\usepackage[font={scriptsize}]{caption}
\usepackage{bbm}
\usepackage{amsrefs}
\usepackage{mathtools}

\usepackage[left=1in, right=1in, top=1.1in,bottom=1.1in]{geometry}
\newcommand{\dw}{\dot{W}}

\newcommand{\D}{\mathbb D}
\newcommand{\E}{\mathbb E}
\newcommand{\R}{\mathbb R}

\newcommand{\cac}{\mathcal C}

\newcommand{\ch}{\mathcal H}
\newcommand{\ci}{\mathcal I}
\newcommand{\cj}{\mathcal J}

\newcommand{\cl}{\mathcal L}
\newcommand{\cm}{\mathcal M}

\newcommand{\cq}{\mathcal Q}

\newcommand{\crr}{\mathcal R}

\newcommand{\al}{\alpha}

\newcommand{\ga}{\gamma}

\newcommand{\ka}{\kappa}
\newcommand{\la}{\lambda}
\newcommand{\laa}{\Lambda}
\newcommand{\om}{\omega}

\newcommand{\si}{\sigma}
\newcommand{\te}{\theta}

\newcommand{\vp}{\varphi}

\newcommand{\lp}{\left(}
\newcommand{\rp}{\right)}
\newcommand{\lc}{\left[}
\newcommand{\rc}{\right]}

\newcommand{\lla}{\left\langle}
\newcommand{\rra}{\right\rangle}

\newtheorem{theorem}{Theorem}[section]

\newtheorem{corollary}[theorem]{Corollary}

\newtheorem{definition}[theorem]{Definition}

\newtheorem{assumption}[theorem]{Assumption}
\newtheorem{lemma}[theorem]{Lemma}

\newtheorem{proposition}[theorem]{Proposition}

\theoremstyle{remark}
\newtheorem{remark}[theorem]{Remark}

\theoremstyle{remark}
\newtheorem{example}[theorem]{Example}

\newcommand{\bean}{\begin{eqnarray*}}
\newcommand{\eean}{\end{eqnarray*}}
\newcommand{\ben}{\begin{enumerate}}
\newcommand{\een}{\end{enumerate}}
\newcommand{\beq}{\begin{equation}}
\newcommand{\eeq}{\end{equation}}

\begin{document}

\def\Pro{{\mathbb{P}}}
\def\E{{\mathbb{E}}}
\def\e{{\varepsilon}}
\def\veps{{\varepsilon}}
\def\ds{{\displaystyle}}
\def\nat{{\mathbb{N}}}
\def\Dom{{\textnormal{Dom}}}
\def\dist{{\textnormal{dist}}}
\def\R{{\mathbb{R}}}
\def\O{{\mathcal{O}}}
\def\T{{\mathcal{T}}}
\def\Tr{{\textnormal{Tr}}}
\def\sgn{{\textnormal{sign}}}
\def\I{{\mathcal{I}}}
\def\A{{\mathcal{A}}}
\def\H{{\mathcal{H}}}
\def\S{{\mathcal{S}}}

\title
[Stochastic heat equation with non-Lipschitz reaction term]
{Regularity of the law of solutions to the stochastic heat equation with non-Lipschitz reaction term}

\author[M. Salins \and S. Tindel ]
{Michael Salins \and Samy Tindel}

 \address{
 Michael Salins:
 Department of Mathematics and Statistics,
 Boston University,
 665 Commonwealth Ave,
 Boston, MA, 02215, 
 USA
 }
 \email{msalins@bu.edu}

 \address{Samy Tindel:
 Department of Mathematics,
Purdue University,
150 N. University Street,
W. Lafayette, IN 47907,
USA.}
\email{stindel@purdue.edu}

%\thanks{}%
%\date{}%
% ----------------------------------------------------------------
\maketitle
\begin{abstract}
We prove the existence of density for the solution to the multiplicative semilinear stochastic heat equation on an unbounded spatial domain, with drift term satisfying a half-Lipschitz type condition. The methodology is based on a careful analysis of differentiability for a map defined on weighted functional spaces.
\end{abstract}

\begin{comment}
\begin{keyword}[class=MSC2020]
\kwd[Primary ]{60F10}
\kwd{60H15}
\kwd[; secondary ]{35R60}
\end{keyword}
\end{comment}
%
%\begin{keyword}
%\kwd{stochastic reaction-diffusion equation}
%\kwd{large deviations}
%\kwd{stochastic partial differential equation}
%\end{keyword}
%
%\end{frontmatter}
%%%%%%%%%%%%%%%%%%%%%%%%%%%%%%%%%%%%%%%%%%%%%%
%% Please use \tableofcontents for articles %%
%% with 50 pages and more                   %%
%%%%%%%%%%%%%%%%%%%%%%%%%%%%%%%%%%%%%%%%%%%%%%
%\tableofcontents

%%%%%%%%%%%%%%%%%%%%%%%%%%%%%%%%%%%%%%%%%%%%%%
%%%% Main text entry area:

\section{Introduction} \label{S:intro}

In this paper we consider a stochastic heat equation on $\R^{d}$ of the form
\begin{equation} \label{eq:SPDE}
  \frac{\partial u}{\partial t}(t,x) = \frac{1}{2} \Delta u(t,x) + f(u(t,x)) + \sigma(u(t,x)) \, \dot{W}(t,x) \, ,
\qquad t\in [0,T], \, x \in\R^{d}.
\end{equation}
In \eqref{eq:SPDE}, $\Delta$ denotes the Laplace operator, $\dw$ is a centered Gaussian noise which is white in time and whose covariance function satisfies a standard assumption called strong Dalang condition (see Assumption~\ref{assum:dalang} below for a precise statement). The coefficient $\si$ in~\eqref{eq:SPDE} is supposed to be differentiable with bounded derivative. Our equation deviates from the standard setting for stochastic \textsc{pde}s due to the drift coefficient $f$. This coefficient is only assumed to verify a mild damping condition, that is we suppose that $f$ is continuously differentiable and that $f'$ is upper bounded by a constant $\ka\in\R$:
\begin{equation}\label{f0}
f'(u) \leq \kappa , \quad\text{ for all }\quad u \in \mathbb{R} .
\end{equation}
This condition will be referred to as half-Lipschitz in the sequel. As a motivating example, any odd degree polynomial with a negative leading coefficient such as $f(u) = -u^3 + u$ will satisfy \eqref{f0}.
Under this setting, we investigate the law of the random field mild solution to~\eqref{eq:SPDE}, $u(t,x)$, at a fixed time $t>0$ and a fixed point in space $x \in \mathbb{R}^d$. We prove using Malliavin calculus that the law has a density.

Stochastic \textsc{pde}s have been primarily considered for globally Lipschitz continuous coefficients $f$ and $\si$ (see e.g~\cite{dalang,Wa}). However, since multiple relevant physical systems involve polynomial type nonlinearities, a substantial amount of effort has been devoted to that case over the past decades. Among those contributions, one can single out the following:
\begin{enumerate}[wide, labelwidth=!, labelindent=0pt, label={(\alph*)}]
\setlength\itemsep{.05in}

\item
The case of $x\in D$ (with $D$ a bounded domain in $\R^{d}$), and polynomial nonlinearities $f$ with negative leading terms, has been investigated in~\cite{cerrai,mnq-2013}. The techniques use localization arguments based on stopping time methods and a priori bounds. The papers~\cite{cerrai,mnq-2013} are all handling the case of a colored noise $\dw$ which can accommodate stochastic integrals without a need for renormalization.

\item
The case of a stochastic heat equation defined on an unbounded spatial domain $\mathbb{R}^d$ with $f$ satisfying
\eqref{f0} and with at most polynomial growth was first investigated independently by Iwata \cite{iwata} and Brzezniak and Peszat \cite{BrzPe}. Unlike in the bounded domain setting, the solutions to \eqref{eq:SPDE} are unbounded in space if $\sigma$ is bounded away from zero. Specifically, for any $t>0$, $\Pro(\sup_{x \in \mathbb{R}^d} |u(t,x)| = +\infty) =1$.

%\item
%Recently, several authors have demonstrated that global solutions exist to the heat equation \eqref{eq:SPDE} for some $f$ that grow superlinearly and that do not satisfy a half-Lipschitz condition like \eqref{f0}. Finite time explosion is characterized by the Osgood condition in this case. If there exists $c>0$ such that $\int_c^\infty \frac{1}{f(u)}du <+\infty$ then solutions explode in finite time and if $\int_c^\infty \frac{1}{f(u)}du = +\infty$ for all $c>0$ then global solutions exist. An example of a superlinear force that does not cause finite explosion is $f(u) = u \log(u)$. \cite{DKZ,FN2021,FN2022,BG,S-2022,SZ,S-2021}

\item
The variational methods of R\"ockner and collaborators can be applied when the perturbing noise is trace-class, so that It\^o formula methods are available \cite{MarRo,PreRo,LiuRo}. This theory allows for \textsc{spde}s that are not semilinear, such as porous medium equations, but precludes  rough perturbations like space-time white noise.

\item
In case of a space-time white noise $\dw$ (or even a spatial white noise if $x\in\R^{d}$ with $d\ge 2$), renormalization tools are in order. We cannot list all the relevant contributions in this direction. Let us just mention~\cite{Ha} for the celebrated \textsc{kpz} equation and~\cite{MW} for the $\Phi_{4}^{3}$ model. Notice that most of those systems only admit an additive noise, and that the current techniques only yield local (in time) solutions.

  \end{enumerate}

%The current contribution can be seen as a step towards a better understanding of the solution $u(t,x)$ to equation~\eqref{eq:SPDE} when the coefficient $f$ is only half-Lipschitz. Relying on a delicate analysis of the differentiation properties of both maps $\ci$ and $\cm$, we will prove that for all $(t,x)\in[0,T]\times\R^{d}$, the random variable $u(t,x)$ is Malliavin differentiable. Then under standard non degeneracy assumptions on the coefficient $\si$ in~\eqref{eq:SPDE}, we shall obtain the existence of a density for $u(t,x)$. This is achieved thanks to a deeper quantitative analysis of the Malliavin derivatives for all terms in our stochastic \textsc{pde}, expressed in terms of the covariance function of the noise $\dw$. Those estimates are interesting in their own right.

Studies of densities for stochastic processes in non-Markovian settings have been one of the great achievements of Malliavin calculus. However, due to a methodology based on differentiation and integration by parts, Malliavin calculus results usually require smooth and bounded coefficients in differential systems like \eqref{eq:SPDE}. This is certainly the case in classical references concerning stochastic \textsc{pde}s \cite{BP,MMS,nualart-book} or systems driven by a fractional Brownian motion~\cite{BNOT,CHLT}. A more recent trend has been to adapt the integration by parts technology to settings with little regularity or less restrictive growth assumptions. One can quote the following studies, which are close in spirit to our own contribution:

\begin{enumerate}[wide, labelwidth=!, labelindent=0pt, label={(\alph*)}]
\setlength\itemsep{.05in}

\item
The article \cite{HKY} deals with a stochastic differential equation driven by an additive Brownian motion, whose drift coefficient lies in a fractional Sobolev space of the form $W^{\ga,p}$ (with a regularity parameter $\ga\in(0,1]$). The computations therein combine Malliavin calculus and Girsanov transform tools.

\item
For stochastic differential equations driven by a fractional Brownian motion let us mention the paper~\cite{FZ}, which handles the case of a H\"older drift. This is achieved thanks to a smart limiting procedure taken on Euler schemes. More recently, the preprint~\cite{LPS} explores densities for a drift coefficient $f$ which has linear growth and satisfies a mild damping condition. The main tools in~\cite{LPS} is Girsanov's transform, again due to the fact that an equation with additive noise is considered. The density is then analyzed by importing arguments from the regularization by noise literature and investigating a functional for a fractional bridge.

\item
In \cite{mnq-2013,mq-2019} the authors consider a \textsc{spde} of the form~\eqref{eq:SPDE}, satisfying an assumption which is similar to~\eqref{f0}. The main difference between this setting  and ours is twofold: first~\cite{mnq-2013,mq-2019} focuses on spatial variables in bounded sets of $\R^{d}$, while our result is concerned with $x$ in the whole space $\R^{d}$. Then~\cite{mnq-2013,mq-2019} is restricted to coefficients $f$ in~\eqref{eq:SPDE} having polynomial growth, while we can reach exponential growth in the current paper. Notice that in~\cite{mq-2019} the strategy is based on a localization procedure relying on Lipschitz approximations of the drift coefficient $f$. This method is ruled out in our unbounded domain setting. Indeed, in our case the field $\{u(t,x); \, x \in \R^{d}   \}$ is unbounded for any fixed $t>0$, even if $f$ is Lipschitz. The boundedness of $u(t,\cdot)$ whenever $f$ is Lipschitz was a crucial ingredient in~\cite{mq-2019}.

\end{enumerate}

\noindent
As one can see, our result is thus the first one establishing existence of density for a \textsc{spde} with drift whose first derivative is unbounded and that is defined on a noncompact domain. On top of this novel aspect, we believe that our method of proof is applicable to other settings. In some subsequent publication we plan to apply the techniques developed here to the renormalized frameworks mentioned above.

{In future work, we also wish to remove the growth restriction on $f$. While many previous works restricted the growth rate of $f$ to polynomial growth like $|f(u)| \leq C(1 + |u|^p)$ \cite{iwata,BrzPe,cerrai,LiuRo}, we allow $f$ to grow as fast as $|f(u)|\leq K e^{K|u|^\nu}$ for any $K,\nu>0$ like in \cite{s-2021}. The exponential growth restriction is helpful for proving that the integrals $\int_0^t \int_{\mathbb{R}^d}G(t-s,x-y)f(u(s,y))dyds$  in the mild solution are well-defined. These growth restrictions do not seem necessary, however, and in future work we hope to prove that the half-Lipschitz condition on $f$ \eqref{f0} along with appropriate assumptions on $\sigma$ and the $\dot{W}$, is sufficient to guarantee the existence and uniqueness of mild solutions and the existence of a density. Such a generalization requires sensitive analysis of the spatial growth rates of solutions and is outside of the scope of the current manuscript.}

As mentioned above,  the solutions to \eqref{eq:SPDE} with $x\in\R^{d}$ are unbounded and heat equations enjoy infinite propagation speed. Therefore the localization arguments that are invoked in the bounded domain case \cite{mnq-2013,mq-2019} cannot be applied to the unbounded domain setting. To investigate properties of unbounded solutions, many researchers have introduced a spatial weight. For example, Iwata \cite{iwata} and Brzezniak and Peszat \cite{BrzPe} used exponential weights $\sup_x e^{-\lambda |x|} |u(t,x)|$. The choice of exponential weights, unfortunately, introduced a polynomial growth restriction in the literature.
With this observation in mind, the first author of this paper proposed in~\cite{s-2021} a new method to handle equations like~\eqref{eq:SPDE}.
Roughly speaking, in this paper and in \cite{s-2021} we use polynomial weights $\sup_{x \in \mathbb{R}^d} \frac{|u(t,x)|}{1 + |x-x_0|^\theta}$ for arbitrarily small $\theta>0$. This choice of weights allows to prove the main results for superlinear half-Lipschitz reaction terms that grow as fast as $|f(u)| \leq K \exp(K|u|^\gamma)$ for any $K,\gamma>0$.

In order to explain how we obtained the existence of a density for the solution to~\eqref{eq:SPDE}, let us give a few details  about the approach in~\cite{s-2021}.
The scheme therein basically splits the dynamics in two pieces: first a stochastic map $\ci$ defined for a random field $\vp$ by
\begin{equation*}
\ci(t,x) = \ci^\varphi(t,x) = \int_0^t \int_{\mathbb{R}^d} G(t-s,x-y) \varphi(s,y)W(ds \, dy) \, ,
\end{equation*}
where $G(t,x):= (2 \pi)^{-\frac{d}{2}} \exp (- |x|^2/2 )$ is the standard Gaussian heat kernel.
This map is properly introduced in~\eqref{eq:def-Phi} below. The second piece of our dynamics is a deterministic map called $\cm$ (see Definition~\ref{def:M}) given for a continuous function $z$ defined on $[0,T]\times\R^{d}$ as the solution of the following integral equation
\begin{equation}\label{f1}
\mathcal{M}(z)(t,x) = \int_0^t \int_{\mathbb{R}^d} G(t-s,x-y) f(\mathcal{M}(z)(s,y)) \, dyds + z(t,x).
\end{equation}
The crucial point in~\cite{s-2021} is that, despite the fact that $f$ in~\eqref{f1} in not globally Lipschitz continuous, the map $\cm$ is globally Lipschitz continuous on weighted spaces of continuous functions on $[0,T]\times\R^{d}$. Thanks to some thorough estimates for both $\ci$ and $\cm$ and a Yosida type approximation procedure for the function $f$, one can prove existence and uniqueness for mild solutions to~\eqref{eq:SPDE}.
More specifically, the \textit{mild solution} of \eqref{eq:SPDE} is defined to be a random field that solves the integral equation
\begin{align} \label{eq:intro-mild}
  u(t,x) = &\int_{\mathbb{R}^d} G(t-s,x-y) u(0,y)dy + \int_0^t \int_{\mathbb{R}^d} G(t-s,x-y) f(u(s,y))dyds \nonumber\\
  &+ \int_0^t \int_{\mathbb{R}^d} G(t-s,x-y) \sigma(u(s,y))W(ds \, dy).
\end{align}
Letting $U_0(t,x) = \int_{\mathbb{R}^d} G(t-s,x-y) u(0,y)dy$ and $\ci^\varphi(t,x)$ and $\cm$ be the maps defined above, equation~\eqref{eq:intro-mild} can be recast as
\begin{equation}
  u(t,x) = \mathcal{M} \left( U_0 + \ci^{\sigma(u)} \right)(t,x).
\end{equation}
The existence and uniqueness of the solution to this equation was then established in \cite{s-2021} via a Picard iteration scheme. Namely, we can recursively define
 \begin{equation} \label{eq:intro-Picard}
   u_0(t,x) = U_0(t,x), \ \ \ \ \ u_{n+1}= \mathcal{M} (U_0 + \ci^{\sigma(u_n)}).
 \end{equation}
By properly bounding both maps $\ci^{\vp}$ and $\cm$, the existence and uniqueness for equation~\eqref{eq:SPDE} is proved thanks to a fixed point argument.

We can now explain our global method for the existence of density result and outline the structure of our paper.
First in Section \ref{sec:approach} we introduce the main assumptions and recall the existence and uniqueness results for \eqref{eq:SPDE} from \cite{s-2021}, as well as Malliavin calculus results from~\cite{nualart-book}.
Then we proceed to prove the Malliavin differentiability of $u(t,x)$ via the approximation scheme~\eqref{eq:intro-Picard}, and also by studying the Malliavin differentiability of the maps $\ci^\varphi$ and $\cm$. To begin with, Section \ref{sec:mall-diff-M} proves Proposition \ref{prop:M-Malliavin}. This result states that if a random field $z : [0,T] \times \mathbb{R}^d \times \Omega \to \mathbb{R}$ has the property that $z(t,x)$ is Malliavin differentiable for any $(t,x) \in [0,T]\times \mathbb{R}^d$ and additionally that $z(t,x)$ and $Dz(t,x)$ satisfy certain polynomial growth assumptions in the spatial variable, then the random field $\mathcal{M}(z)(t,x)$ is also Malliavin differentiable for any $(t,x) \in [0,T]\times \mathbb{R}^d$. Furthermore, Proposition \ref{prop:M-Malliavin} establishes a weighted supremum norm bound which holds with probability one:
\begin{equation}
  \sup_{t\in [0,T]} \sup_{x \in \mathbb{R}^d} \frac{|D\mathcal{M}(z)(t,x)|_{\H_T}}{1 + |x-x_0|^\theta} \leq K \sup_{t \in [0,T]} \sup_{x \in \mathbb{R}^d} \frac{ |Dz(t,x)|_{\H_T}}{1 + |x-x_0|^\theta} \, ,
\end{equation}
where $\ch_{T}$ is the natural Cameron-Martin space related to our colored noise (see~\eqref{eq:H_T} below for a proper definition of the inner product in $\ch_{T}$).

The Malliavin differentiability of the stochastic integrals $\ci^\varphi$ when $\varphi$ is Malliavin differentiable is a standard result from the theory of Malliavin calculus (see Proposition \ref{prop:mall-diff-for-stoch-intg} below). In Section~\ref{sec:mall-diff-I} we prove that certain moment estimates of weighted supremum norms, that applied to stochastic integrals $\ci^\varphi(t,x)$ when $\varphi$ is real-valued, will also hold in the case where $\varphi$ is Hilbert-space valued. Specifically, we apply these results to derive estimates on the Malliavin derivatives of the stochastic integral terms.
In Section \ref{sec:mall-diff} we apply these Malliavin differentiability results about $\ci^\varphi$ and $\mathcal{M}$ to the recursively defined Picard iteration scheme introduced in \eqref{eq:intro-Picard}. In particular, this allows us to prove that  $u_n(t,x)$ is Malliavin differentiable for all $n \in \mathbb{N}$, $t \in [0,T]$, and $ x \in \mathbb{R}^d$. Furthermore, we prove that for any $p>0$ and $T>0$,
\begin{equation}
  \sup_n \sup_{x_0 \in \mathbb{R}^d} \E \left| \sup_{t \in [0,T]} \sup_{x \in \mathbb{R}^d} \frac{|Du_n(t,x)|_{\H_T}}{1 + |x-x_0|^\theta} \right|^p< +\infty \, .
\end{equation}
In particular, these weighted supremum moment bounds guarantee that $\sup_n E|Du_n(t,x)|_{\H_T}^2$ is finite for any fixed $t>0$ and $x \in \mathbb{R}^d$. The classical result \cite[Lemma 1.2.3]{nualart-book} then guarantees that $u(t,x)$ is Malliavin differentiable.

Finally, in Section \ref{sec:positivity} we prove that the Malliavin derivative of the mild solution is positive almost surely, that is $\Pro(\|Du(t,x)\|_{\H_T}>0)=1$. By \cite[Theorem 2.1.2]{nualart-book}, this positivity property implies that the law of $u(t,x)$ is absolutely continuous with respect to Lebesgue measure. We prove the positivity by constructing a particular family of deterministic test functions $h_{t,x,\delta} \in \H_T$ and proving that, with probability one, the directional Malliavin derivative $D_{h_t,x,\delta}u(t,x)$ is non-negative for some small (and random) $\delta>0$. This analysis involves writing the directional derivatives in a mild form
\[D_h u(t,x) = \left< \phi_{t,x},h \right>_{\H_T} + A_h(t,x) + B_h(t,x) \, ,
\]
where $A_h(t,x)$ is  a Lebesgue integral and $B_h(t,x)$ is a stochastic integral. We prove that when $\delta$ is sufficiently small, the integral terms are much smaller than the leading term, implying that the directional Malliavin derivative is non-negative.

\begin{comment}
In Section \ref{sec:approach} we introduce the main assumptions and recall the existence and uniqueness results for \eqref{eq:SPDE} from \cite{s-2021} and Malliavin calculus results from \cite{nualart-book}. In Section \ref{sec:mall-diff-M} we prove that the composition $\mathcal{M}(z)$ is Malliavin differentiable if $z$ satisfies appropriate assumptions. In Section \ref{sec:mall-diff-I} we establish that the mild solution $u(t,x)$ is Malliavin differentiable for any $t>0$, and $x \in \mathbb{R}^d$. Finally, in Section \ref{sec:positivity} we prove the positivity of the Malliavin derivative.
\end{comment}

\begin{comment}
Marinelli, Quer-Sardanyons, and E. Nualart have results in the case of a bounded spatial domain and additive noise ($\sigma \equiv 1$) \cite{mnq-2013}. They have a few extra restrictions on $f$ in addition to the half-Lipschitz condition. For example, they require that $f$ and its derivatives grow no faster than polynomials. I do not know if these extra restrictions are necessary, or if they just simplified the analysis. A more recent preprint by Marinelli and Quer-Sardanyons \cite{mq-2019} considers the setting of a bounded spatial domain with multiplicative noise and also includes polynomial growth conditions and other restrictions.
\end{comment}

\section{Approach to existence and uniqueness} \label{sec:approach}
In this section we will summarize the method employed in \cite{s-2021} in order to solve an equation like \eqref{eq:SPDE} with a half-Lipschitz reaction term. The method is based on a fixed point argument in an appropriate weighted H\"older space. We also include a minimal set of Malliavin calculus tools necessary to carry out our main computations.

\subsection{Functional space, assumptions and existence result}
We start by defining the weighted function spaces which will be used throughout the paper.
\begin{definition}\label{def:weighted-cont-space}
  Let $\theta>0$ be a positive parameter {and let $x_0 \in \mathbb{R}^d$}. The space $\cac_{\theta,x_0}([0,T]\times \mathbb{R}^d)$ designates the set of continuous functions
  \begin{equation}
    \left\{ z \in \cac([0,T]\times \mathbb{R}^d) : \lim_{|x| \to \infty} \sup_{t \in [0,T]} \frac{|z(t,x)|}{1 + |x-x_0|^\theta} = 0 \right\}.
  \end{equation}
  The space is endowed with the weighted supremum norm
  \begin{equation}
    |z|_{\cac_{\theta,x_0}([0,T]\times\mathbb{R}^d)} : = \sup_{t \in [0,T]} \sup_{x \in \mathbb{R}^d} \frac{|z(t,x)|}{1 + |x-x_0|^\theta}.
  \end{equation}
  %\hb{While reading the paper I was sometimes confused by the notation $C_{\te}$, which looks like a constant. Would it be ok to change it for $\cac_{\te}$?}
  For fixed $\theta$, the spaces $\cac_{\theta,x_0}([0,T]\times \mathbb{R}^d)$ all coincide, but it is convenient to use different centers of the weight $x_0$.
\end{definition}

Next we state the assumptions on the stochastic noise $\dot{W}$. All of the random variables below are defined on a complete probability space $(\Omega, \mathcal{F}, \Pro)$ equipped with a filtration $\{\mathcal{F}_t: t \geq 0\}$.

\begin{assumption} \label{assum:dalang}
  The noise $\dot{W}$ in \eqref{eq:SPDE} is a centered Gaussian spatially homogeneous noise which is white in time. There exists a positive and positive definite function $\Lambda$ such that formally
  \begin{equation}\label{eq:def-covariance-formal}
    \E[\dot{W}(t,x) \dot{W}(s,y)] = \delta(t-s)\Lambda(x-y).
  \end{equation}
  In the above expression, $\delta$ is the Dirac measure. The Fourier transform of $\Lambda$ is a measure $\mu$ and we assume that there exists $\eta \in (0,1)$ such that
  \begin{equation}\label{eq:strong-dalang}
    \int_{\mathbb{R}^d} \frac{1}{1 + |\xi|^{2(1-\eta)}} \mu(d\xi) < +\infty.
  \end{equation}
\end{assumption}

The multiplicative noise coefficient $\sigma$ in \eqref{eq:SPDE} satisfies  standard differentiability and nondegeneracy assumptions.

\begin{assumption} \label{assum:sig}
  The noise coefficient $\sigma: \mathbb{R} \to \mathbb{R}$ is differentiable and its derivative is uniformly bounded.  Moreover, we assume that there exists $\alpha>0$ such that
  \begin{equation} \label{eq:sigma-non-degen}
    \sigma(u) \geq \alpha , \quad\text{ for all }\quad u \in \mathbb{R}.
  \end{equation}
\end{assumption}

As mentioned in the introduction, our system \eqref{eq:SPDE} departs from the standard stochastic \textsc{pde} setting due to the drift coefficient $f$. Namely we only suppose that $f$ in \eqref{eq:SPDE} satisfies a half-Lipschitz condition,  is differentiable, and obeys a very mild growth condition. This is summarized in the assumption below.
\begin{assumption} \label{assum:f}
  The reaction term $f:\mathbb{R} \to \mathbb{R}$ is continuously differentiable. Moreover, there exists $\kappa \in \mathbb{R}$ such that the derivative is uniformly bounded from above
  \begin{equation} \label{eq:f-deriv-bound}
    f'(u) \leq \kappa , \quad\text{ for all }\quad u \in \mathbb{R} .
  \end{equation}
  We further assume that there exist $K>0$, $\nu>0$ such that
  \begin{equation} \label{eq:f-exp-bound}
    |f'(u)| \leq K \exp \left(K |u|^\nu \right)
  \end{equation}
\end{assumption}
Notice that the upper bound on the first derivative \eqref{eq:f-deriv-bound} implies that $f: \mathbb{R} \to \mathbb{R}$ is \textit{half-Lipschitz}, meaning that for any $u_1>u_2$,
  \begin{equation} \label{eq:half-Lip}
    f(u_1) - f(u_2) \leq \kappa (u_1 - u_2) .
  \end{equation}
We now label a standard assumption for the initial condition $u_{0}$ for our equation of interest.
  \begin{assumption} \label{assum:init-cond}
    The initial condition for \eqref{eq:SPDE} is continuous and uniformly bounded, meaning that there exists $M>0$ such that
    \begin{equation} \label{eq:init-cond-bound}
      \sup_{x \in \mathbb{R}^d} |u_0(x)| \leq M.
    \end{equation}
  \end{assumption}

%\subsection{Existence and uniqueness result}

  Equation \eqref{eq:SPDE} is solved in the so-called mild sense. Namely an adapted process $u$ is said to solve~\eqref{eq:SPDE} with initial condition $u_{0}$ if for all $(t,x)\in[0,T]\times \mathbb{R}^d$ we have
  \begin{align}\label{eq:mild-solution}
  u(t,x) = &\int_{\R^{d}} G(t,x-y) u_0(y)dy + \int_0^t \int_{\R^{d}} G(t-s,x-y) f(u(s,y))dyds \nonumber \\
  &+ \int_0^t \int_{\R^{d}} G(t-s,x-y)\sigma(u(s,y)) \, W(ds \, dy) \, ,
\end{align}
where $G(t,x)$ is the heat kernel related to the Laplace operator in \eqref{eq:SPDE}. In \eqref{eq:mild-solution}, the stochastic integral is interpreted in the It\^o sense (see \cite{dalang} for further details).
With this notion of solution and under the set of assumptions spelled out above, we now recall the main existence and uniqueness result from \cite{s-2021}.

  \begin{theorem}[Theorem 2.6 of \cite{s-2021}]
    Suppose Assumptions \ref{assum:sig}--\ref{assum:init-cond} are satisfied. Then there exists a unique mild solution to \eqref{eq:SPDE}. This solution lives in the space $C_\theta([0,T]\times \mathbb{R}^d)$ introduced in Definition~\ref{def:weighted-cont-space}.
  \end{theorem}

\subsection{Methodology}
  In this section we review the methods used to solve \eqref{eq:SPDE} in \cite{s-2021}. Those tools will also play a prominent role in analyzing the Malliavin derivative of the solution.

  \subsubsection{Yosida approximations}\label{sec:yosida}
A crucial ingredient in the analysis of equation \eqref{eq:mild-solution} is based on Yosida approximations for the  nonlinear forcing term $f$ satisfying Assumption \ref{assum:f}. That is
for any function $f$ satisfying \eqref{eq:f-deriv-bound} or \eqref{eq:half-Lip} it is easily seen (see \cite[Proposition 2.4]{s-2021}) that
\begin{equation} \label{eq:f=phi+line}
  f(u) = \phi(u) + \kappa u \, ,
\end{equation}
where $\phi$ is non-increasing.
For  $\phi:\mathbb{R} \to \mathbb{R}$ that is non-increasing, we define the Yosida approximations for $\lambda>0$ by
\begin{equation}\label{eq:def-yosida}
  \phi_\lambda(u):= \frac{1}{\lambda}(J_\lambda(u) - u) \text{ where } J_\lambda(u) = (I - \lambda \phi)^{-1}(u).
\end{equation}
The family $\{\phi_\lambda: \lambda>0\}$ is intended to be a smooth approximation of $\phi$ under monotonicity conditions.
We now summarize some properties of the Yosida approximations, taken from   \cite[Appendix D]{dpz-book}.
\begin{lemma}
Let $\phi:\mathbb{R} \to \mathbb{R}$ be a differentiable non-increasing function and let $\{\phi_\lambda: \lambda>0\}$ be its Yosida approximations defined by \eqref{eq:def-yosida}. Then the following are true.
\begin{enumerate}[wide, labelwidth=!, labelindent=0pt, label=\emph{(\roman*)}]
\setlength\itemsep{.02in}
    \item $|\phi_\lambda(u_1) - \phi_\lambda(u_2)| \leq \frac{2}{\lambda} |u_1-u_2|$, for $u_{1},u_{2}\in\R$ and all $\la>0$.
    \item $|\phi_\lambda(u)| \leq |\phi(u)|$, for $u\in\R$ and all $\la>0$.
    \item $u \mapsto \phi_\lambda(u)$ is nonincreasing, for all $\la>0$.
    \item $\lim_{\lambda \to 0} \phi_\lambda(u) = \phi(u)$, for all $u\in\R$.
    \item $\lim_{\lambda \to 0} \phi_\lambda'(u) = \phi'(u)$, for all $u\in\R$.
  \end{enumerate}
\end{lemma}

These properties of Yosida approximations are easily translated into approximations for the half-Lipschitz function $f$.
\begin{lemma} \label{lem:f-Yosida-properties}
Let $f$ satisfy Assumption \ref{assum:f} so that $f$ satisfies the decomposition \eqref{eq:f=phi+line}. Define a family $\{f_\lambda : \lambda>0\}$ by
\begin{equation} \label{eq:f-Yosida}
  f_\lambda(u) = \phi_\lambda(u) + \kappa u
\end{equation}
where $\phi_\lambda$ are Yosida approximations of the non-increasing function $\phi$. Then $f_\lambda$ satisfies the following properties.
\begin{enumerate}[wide, labelwidth=!, labelindent=0pt, label=\emph{(\roman*)}]
\setlength\itemsep{.02in}
    \item $|f_\lambda(u_1) - f_\lambda(u_2)| \leq \left(\frac{2}{\lambda} + \kappa \right) |u_1-u_2|$, for $u_{1},u_{2}\in\R$ and all $\la>0$.
    \item $|f_\lambda(u)| \leq (1 + 2\kappa)|f(u)|$, for $u\in\R$ and all $\la>0$.
    \item\label{item:half-lip-f-lambda}
    $(f_\lambda(u_1) - f_\lambda(u_2)) \, \sgn(u_1-u_2) \leq \kappa |u_1- u_2|$, for $u_{1},u_{2}\in\R$ and all $\la>0$.
    \item $\lim_{\lambda \to 0} f_\lambda(u) = f(u)$, for all $u\in\R$.
    \item $\lim_{\lambda \to 0} f_\lambda'(u) = f'(u)$, for all $u\in\R$.
  \end{enumerate}
\end{lemma}

\subsubsection{Mapping $\cm$}\label{sec:mapping-M}
The second ingredient we wish to highlight in the study of~\eqref{eq:SPDE}   is the introduction of a functional mapping $\mathcal{M}:\cac_{\theta,x_0}([0,T]\times \mathbb{R}^d) \to \cac_{\theta,x_0}([0,T]\times \mathbb{R}^d)$.

  \begin{definition} \label{def:M}
    For a continuous function $z \in \cac_{\theta,x_0}([0,T]\times \mathbb{R}^d)$ let $\mathcal{M}(z)$ be the solution to the following equation
    \begin{equation} \label{eq:M-def}
      \mathcal{M}(z)(t,x) = \int_0^t \int_{\mathbb{R}^d} G(t-s,x-y) f(\mathcal{M}(z)(s,y)) \, dyds + z(t,x).
    \end{equation}
  \end{definition}

\begin{remark}
In order to prove existence of the map  $\cm$ one uses an approximating sequence $\{\cm_{\la}; \, \la >0\}$ defined as in \eqref{eq:M-def}, with $f$ replaced by its Yosida approximation $f_{\la}$ given in~\eqref{eq:def-yosida}. Then some a priori estimates on $\cm_{\la}(z)$ are provided in~\cite{s-2021}. Those estimates allow to conclude the existence part, thanks to some compactness arguments.
\end{remark}

  With our Malliavin calculus considerations in mind, we formulate a time and space inhomogenous version of Theorem 5.6 of \cite{s-2021}. To this aim, we consider $\varphi:[0,T]\times \mathbb{R}^d \times \mathbb{R}$ and assume that $\varphi$ is uniformly half-Lipschitz in the third argument. This means that there exists $\kappa \in \mathbb{R}$ such that for any $t \in [0,T]$, $x \in \mathbb{R}^d$, and $u_1> u_2 \in \mathbb{R}$,
  \begin{equation} \label{eq:varphi-half-lip}
    \varphi(t,x,u_1) -\varphi(t,x,u_2) \leq \kappa (u_1 - u_2).
  \end{equation}
  {We also impose the growth restriction that there exist $K>0, \nu>0$, $x_{0}\in\R^{d}$ and $\beta \in [0,2)$ such that for any $t \in [0,T]$
  \begin{equation} \label{eq:varphi-growth-restriction}
    |\varphi(t,x, u)| \leq K e^{K(|x-x_0|^\beta + |u|^\nu)}.
  \end{equation}
  }
  {We introduce a new functional mapping $\cl$ in the following way. Given $\varphi$ satisfying \eqref{eq:varphi-half-lip}--\eqref{eq:varphi-growth-restriction} and $z \in C_{\theta,x_0}([0,T]\times\mathbb{R}^d)$, let $\mathcal{L}(z) \in C_{\theta,x_0}([0,T]\times{\mathbb{R}^d})$ denote the solution to
  \begin{equation} \label{eq:varphi-mild}
      \cl(z)(t,x) = \int_0^t \int_{\mathbb{R}^d} G(t-s,x-y) \varphi(s,y,\cl(z)(s,y)) \, dyds + z(t,x).
    \end{equation}
    }
    {The growth restriction \eqref{eq:varphi-growth-restriction} guarantees that the above integral is finite if $\mathcal{L}(z) \in \cac_{\theta,x_0}([0,T]\times \mathbb{R}^d)$ for some $\theta \in (0, 2/\nu)$.}
    {
    \begin{remark}
      The existence of a solution $\mathcal{L}(z) \in C_{\theta,x_0}([0,T]\times \mathbb{R}^d)$ for any $z \in C_{\theta,x_0}([0,T]\times \mathbb{R}^d)$ can be proved via Yosida approximations following the arguments of Theorem 5.2 of \cite{s-2021}. We will prove the existence of Malliavin derivatives that solve \eqref{eq:varphi-mild} in Section \ref{sec:mall-diff-M} below, and we have no need to prove the existence of $\mathcal{L}(z)$ in full generality.  We do need to prove that $\mathcal{L}$ features a global Lipschitz continuity property on the domain where it exists and we will use this property frequently in the sequel.
    \end{remark}}
  \begin{theorem} \label{thm:M-Lip-inhomo}
    Consider a function $\varphi:[0,T]\times \mathbb{R}^d \times \mathbb{R}$ verifying~\eqref{eq:varphi-half-lip}--\eqref{eq:varphi-growth-restriction} and a generic $x_{0}\in\R^{d}$.
    Let $\theta \in (0,2/\nu)$.
    There exists $K=K(T,\theta,\kappa)>0$ such that if  $z_1, z_2 \in \cac_{\theta,x_0}([0,T]\times\mathbb{R}^d)$ and if there exist $\cl(z_1), \cl(z_2) \in \cac_{\theta,x_0}([0,T]\times\mathbb{R}^d)$ that solve \eqref{eq:varphi-mild}, then
    \begin{equation}\label{eq:increments-cl}
      |\cl(z_1) - \cl(z_2)|_{\cac_{\theta,x_0}([0,T]\times\mathbb{R}^d)} \leq K |z_1 - z_2|_{\cac_{\theta,x_0}([0,T]\times\mathbb{R}^d)}.
    \end{equation}
    {The constant $K$ does not depend on the center of the weight $x_0$ and only depends on $\varphi$ through the parameter $\kappa$.}
  \end{theorem}

  \begin{proof}
    Let $v_i(t,x):= \cl(z_i)(t,x) - z_i(t,x)$ for $ i \in \{1,2\}$ and let $\tilde{v}(t,x) = v_1(t,x)  -v_2(t,x)$. The function $\tilde{v}$ is weakly differentiable and
    \begin{equation}
      \frac{\partial \tilde v}{\partial t}(t,x) = \frac{1}{2}\Delta \tilde{v}(t,x) + \varphi(t,x,v_1(t,x) + z_1(t,x)) - \varphi(t,x,v_2(t,x) + z_2(t,x)).
    \end{equation}

    Without loss of generality, we can assume that $\tilde{v}$ is strongly differentiable by approximating $\tilde{v}$ using resolvent operators \cite[Proposition 6.2.2]{cerrai-book}. Let $\rho(x) = (1+ |x-x_0|^2)^{\frac{\theta}{2}}$ be a twice-differentiable weight. Then the quotient $\tilde{q}(t,x) = \frac{\tilde{v}(t,x)}{\rho(x)}$ satisfies
    \begin{align} \label{eq:tilde-q-pde}
      \frac{\partial \tilde q}{\partial t}(t,x) = &\frac{1}{2}\Delta \tilde{q}(t,x) + \nabla \tilde{q}(t,x) \cdot \frac{\nabla \rho(x)}{\rho(x)} + \frac{1}{2}\tilde{q}(t,x)\frac{ \Delta \rho(x)}{\rho(x)} \nonumber\\
      &+ \frac{\varphi(t,x,v_1(t,x) + z_1(t,x)) - \varphi(t,x,v_2(t,x) + z_2(t,x))}{\rho(x)}.
    \end{align}
    %Direct calculations verify that because $\rho(x) = 1  + |x|^\theta$,
%    \begin{equation}
%      \sup_{x \in \mathbb{R}^d} \frac{\Delta \rho(x)}{\rho(x)} <+\infty.
%    \end{equation}

    By the assumption that $\tilde{v} \in \cac_{\theta,x_0}([0,T]\times \mathbb{R}^d)$, the weighted difference $\tilde{q}$ sits in the space $\cac_0([0,T]\times \mathbb{R}^d)$, meaning that $\lim_{|x| \to \infty} \sup_{t \in [0,T]}|\tilde{q}(t,x)| = 0.$ For any $t \in [0,T]$, there exists at least one point $x_t \in \mathbb{R}^d$ where the supremum is attained. Specifically,
   \begin{equation}\label{a0}
      |\tilde{q}(t,x_t)| = \sup_{x \in \mathbb{R}^d} |\tilde{q}(t,x)|.
    \end{equation}
    Furthermore, the upper-left derivative of the supremum is bounded by
    \begin{equation} \label{eq:deriv-of-sup}
      \frac{d^-}{dt} |\tilde{q}(t,\cdot)|_{\cac_0} \leq \frac{\partial \tilde{q}}{\partial t}(t,x_t) \sgn(\tilde{q}(t,x_t)),
    \end{equation}
    where $x_t$ is any maximizer such that relation~\eqref{a0} holds true (see \cite[Proposition 3.5]{s-2021}). Therefore
applying \eqref{eq:deriv-of-sup} and \eqref{eq:tilde-q-pde}, the left derivative above satisfies
    \begin{multline}\label{a1}
      \frac{d^-}{dt} |\tilde{q}(t,\cdot)|_{\cac_0}  \leq  \frac{1}{2}\Delta \tilde{q}(t,x_t) \sgn(\tilde{q}(t,x_t)) + \nabla \tilde{q}(t,x_t) \cdot \frac{\nabla \rho(x_t)}{\rho(x_t)}\sgn(\tilde{q}(t,x_t)) \\
      + \frac{1}{2}\tilde{q}(t,x_t)\frac{ \Delta \rho(x_t)}{\rho(x_t)}\sgn(\tilde{q}(t,x_t)) + \cq_{t} \, ,
    \end{multline}
where we have set
\begin{equation}\label{a11}
\cq_{t} \equiv
\frac{\varphi(t,x_t,v_1(t,x_t) + z_1(t,x_t)) - \varphi(t,x_t,v_2(t,x_t) + z_2(t,x_t))}{\rho(x_t)} \sgn(\tilde{q}(t,x_t))
\end{equation}
We now examine the right hand side of \eqref{a1}.
    \begin{enumerate}[wide, labelwidth=!, labelindent=0pt, label=\textbf{(\roman*)}]
\setlength\itemsep{.02in}
\item
Because $x_t$ is a maximizer or minimizer for $\tilde{q}$, we have
    \begin{equation}\label{a2}
      \nabla \tilde{q}(t,x_t) = 0.
    \end{equation}
\item
    The convexity of a function at a local maximizer or minimizer guarantees that
    \begin{equation}\label{a3}
      \Delta \tilde{q}(t,x_t) \sgn(\tilde{q}(t,x_t)) \leq 0.
    \end{equation}
\item
    Direct calculations verify that $\sup_x \frac{\Delta \rho(x)}{\rho(x)}<+\infty$ so that
    \begin{equation}\label{a4}
      \frac{1}{2}\tilde{q}(t,x_t)\frac{ \Delta \rho(x_t)}{\rho(x_t)} \sgn(\tilde{q}(t,x_t)) \leq C|\tilde{q}(t,x_t)|.
    \end{equation}
\item
    We split the analysis of the $\cq_{t}$ term in~\eqref{a11} into two cases, according to the relation $|\tilde{q}(t,\cdot)|_{\cac_0} > |\tilde{z}(t,\cdot)|_{\cac_0}$ or $|\tilde{q}(t,\cdot)|_{\cac_0} \le |\tilde{z}(t,\cdot)|_{\cac_0}$. Namely let $\tilde{z}(t,x)$ be the weighted difference
    \begin{equation}
      \tilde{z}(t,x) = \frac{z_1(t,x) - z_2(t,x)}{\rho(x)}.
    \end{equation}
    If $|\tilde{q}(t,\cdot)|_{\cac_0} > |\tilde{z}(t,\cdot)|_{\cac_0}$, then
    \begin{equation}
      \sgn(\tilde{q}(t,x_t)) = \sgn(\tilde{q}(t,x_t) + \tilde{z}(t,x_t)) = \sgn(v_1(t,x_t) + z_1(t,x_t) - (v_2(t,x_t) + z_2(t,x_t))).
    \end{equation}
    In this case, \eqref{eq:varphi-half-lip} guarantees that
    \begin{equation}\label{a5}
      \cq_{t}
      \leq \kappa |\tilde{q}(t,\cdot)+\tilde{z}(t,\cdot)|_{\cac_0}
      \leq 2 \kappa |\tilde{q}(t,\cdot)|.
    \end{equation}
   Hence plugging \eqref{a2}-\eqref{a3}-\eqref{a4} and \eqref{a5} into \eqref{a1}, in the case where $|\tilde{q}(t,\cdot)|_{\cac_0} > |\tilde{z}(t,\cdot)|_{\cac_0}$ we get
    \begin{equation}\label{a6}
      \frac{d^-}{dt} |\tilde{q}(t,\cdot)|_{\cac_0} \leq C |\tilde{q}(t,\cdot)|_{\cac_0} ,
    \end{equation}
    where the constant $C$ depends only on $\kappa$ and $\theta$.
%\item\label{item:a7}

    On the other hand if $|\tilde{z}(t,\cdot)|_{\cac_0}> |\tilde{q}(t,\cdot)|_{\cac_0}$, then we cannot get a bound on the left derivative $\frac{d^-}{dt}|\tilde{q}(t,\cdot)|_{\cac_0}$, but this is not a problem because in this case we have an explicit upper bound on $|\tilde{q}(t,\cdot)|_{\cac_0}$ itself. To deal with both of these cases simultaneously, it is convenient to bound the left derivative of
    \begin{equation}\label{a7}
      \max\left\{|\tilde{q}(t,\cdot)|_{\cac_0}, M \right\}, 
      \quad\text{where}\quad M:= \sup_{s \in [0,T]}|\tilde{z}(t,\cdot)|_{\cac_0}.
    \end{equation}
    Specifically, if $|\tilde{q}(t,\cdot)|_{\cac_0}> M$, then the left derivative of $\max\left\{|\tilde{q}(t,\cdot)|_{\cac_0}, M \right\}$ is \eqref{a6}, while if $|\tilde{q}(t,\cdot)|_{\cac_0}\leq M$, then the left derivative of $\max\left\{|\tilde{q}(t,\cdot)|_{\cac_0}, M \right\}$ is $0$.
    
        \end{enumerate}

    From \eqref{a1}-\eqref{a6}-\eqref{a7} and the considerations above, we can see that for any fixed $T>0$ and for any $t \in [0,T]$ we have
    \begin{equation}\label{a8}
      \frac{d^-}{dt} \max \left\{ |\tilde{q}(t,\cdot)|,  M\right\}
      \leq C |\tilde{q}(t,\cdot)|_{\cac_0} \leq  C\max \left\{ |\tilde{q}(t,\cdot)|,  M\right\}.
    \end{equation}
    Using the fact that $\tilde{q}(0,x) \equiv 0$, one can integrate \eqref{a8} in order to get an exponential growth bound:
    \begin{equation}
      \sup_{t \in [0,T]} \max \left\{ |\tilde{q}(t,\cdot)|,  M\right\} \leq M e^{CT}.
    \end{equation}
    Therefore because $M = \sup_{t \in  [0,T]} \sup_{x \in \mathbb{R}^d} |\tilde{z}(t,\cdot)|_{\cac_0}$,
    \begin{equation}
      \sup_{t \in  [0,T]} \sup_{x \in \mathbb{R}^d} |\tilde{q}(t,x)| \leq e^{CT} \sup_{t \in  [0,T]} \sup_{x \in \mathbb{R}^d} |\tilde{z}(t,\cdot)|_{\cac_0}.
    \end{equation}
    The definitions of $\tilde{q}$ and $\tilde{z}$ guarantee that
    \begin{equation}
      \sup_{t \in  [0,T]} \sup_{x \in \mathbb{R}^d} \frac{|v_1(t,x)-v_2(t,x)|}{1 + |x-x_0|^\theta} \leq e^{CT} \sup_{t \in  [0,T]} \sup_{x \in \mathbb{R}^d} \frac{|z_1(t,x)-z_2(t,x)|}{1 + |x-x_0|^\theta}.
    \end{equation}
    Our claim \eqref{eq:increments-cl} follows because $\cl(z_i) = v_i + z_i$.
  \end{proof}

As a particular case of Theorem \ref{thm:M-Lip-inhomo} for a homogeneous function $\vp$, we get the fact that $\cm$ in Definition~\ref{def:M} is a Lipschitz map on $\cac_{\theta,x_0}([0,T]\times\mathbb{R}^d)$ (this was the content of Theorem~5.6 in~\cite{s-2021}). {The existence of $\mathcal{M}$ was proved in Theorem 5.2 of \cite{s-2021}.}

  \begin{proposition}\label{prop:M-Lip}
  Suppose  Assumption \ref{assum:f} is satisfied.
    Let $\theta \in \left(0, \frac{2}{\nu}\right)$  where $\nu$ is from~\eqref{eq:f-deriv-bound} %\hb{(This condition on $\theta$, $\nu$ does not appear in Theorem \ref{thm:M-Lip-inhomo}. Where is it used here?)}.
    For any $z \in \cac_{\theta,x_0}([0,T]\times \mathbb{R}^d)$ there exists a unique solution $\mathcal{M}(z) \in \cac_{\theta,x_0}([0,T]\times\mathbb{R}^d)$ to~\eqref{eq:M-def}. Furthermore, $\mathcal{M}$ is globally Lipschitz continuous. Specifically, there exists a constant $K=K(T,\theta,\kappa)$, depending only on the weight parameter $\theta$ and $\kappa$ from Assumption \ref{assum:f} such that for any two functions $z_1, z_2 \in C_\theta([0,T]\times \mathbb{R}^d)$,
    \begin{equation} \label{eq:M-Lip}
      |\cm(z_{2}) - \cm(z_{1}) |_{\cac_{\theta,x_0}([0,T]\times \mathbb{R}^d)} \leq K \, |z_2 - z_1|_{\cac_{\theta,x_0}([0,T]\times \mathbb{R}^d)}.
    \end{equation}
  \end{proposition}

\subsubsection{Approximation scheme for existence and uniqueness}\label{sec:approx-scheme}
With the preliminary results in Sections \ref{sec:mapping-M} and \ref{sec:yosida}, the Picard iterations approximating the solution of \eqref{eq:mild-solution} are defined as follows in~\cite{s-2021}.

\begin{enumerate}[wide, labelwidth=!, labelindent=0pt, label=\emph{(\roman*)}]

\item
Initiate the iterations by setting
\begin{eqnarray}
\label{eq:U-Picard0}
 U_0(t,x)&: =&
  \int_{\mathbb{R}^d} G(t,x-y)u_0(y)dy \\
  \label{eq:Z-Picard0}
  Z_0(t,x) &:=&
  0
\end{eqnarray}

\item\label{it:from-u-n-to-Z-n}
Given $(u_{n}, Z_{n})$, define
\begin{equation}\label{eq:Z-Picard1}
Z_{n+1}(t,x) := \int_0^t \int_{\R^{d}} G(t-s,x-y)\sigma(u_{n}(s,y))W(ds \, dy) .
\end{equation}

\item\label{it:from-Z-n-to-u-n}
Once $Z_{n+1}$ is introduced, set
\begin{equation}\label{eq:u-Picard}
u_{n+1}(t,x) := \mathcal{M}(U_0 + Z_{n+1})(t,x).
\end{equation}

\end{enumerate}

\noindent
It is proved in~\cite[Theorem 5.2]{s-2021} that the sequences $\{u_{n}; \, n \ge 0\}$ and $\{Z_{n}; \, n \ge 0\}$ converge to the solution of  \eqref{eq:mild-solution}.
Specifically, the following relation hold true for all $p\ge 1$:
\begin{align} \label{eq:Z_n-conv}
  &\lim_{n \to \infty}\sup_{x_0 \in \mathbb{R}^d} \E \left|\sup_{t \in [0,T]} \sup_{x \in \mathbb{R}^d} \frac{|Z_n(t,x) - Z(t,x)|}{1 + |x-x_0|^\theta} \right|^p = 0.\\ \label{eq:u_n-conv}
  &\lim_{n \to \infty} \sup_{x_0 \in \mathbb{R}^d}\E \left|\sup_{t \in [0,T]} \sup_{x \in \mathbb{R}^d} \frac{|u_n(t,x) - u(t,x)|}{1 + |x-x_0|^\theta} \right|^p = 0,
\end{align}
where $u$ solves \eqref{eq:mild-solution} and $Z$ is the stochastic convolution
\begin{equation}\label{eq:Z-limit}
Z(t,x) := \int_0^t \int_{\R^{d}} G(t-s,x-y)\sigma(u(s,y))W(ds \, dy) .
\end{equation}
Also recall that in item~\ref{it:from-Z-n-to-u-n} above, the mapping $\cm$ is defined through a limiting procedure involving the Yosida approximations $f_{\la}$ of $f$.

\subsection{Malliavin Calculus}

This section is devoted to review some elementary notions of Malliavin calculus (mostly borrowed from \cite{nualart-book}). We first recall that our noise $\dw$ is a Gaussian centered field whose covariance is formally given by~\eqref{eq:def-covariance-formal}. One can also look at $\dw$ as a centered Gaussian family $\{W(\vp); \, \vp \in\ch_{T}\}$, where $\ch_{T}$ denotes the Hilbert space with inner product
\begin{equation} \label{eq:H_T}
  \left<\phi, \psi \right>_{\H_T} = \int_0^T \int_{\mathbb{R}^d} \int_{\mathbb{R}^d} \phi(t,y_1)\psi(t,y_2)\Lambda(y_1-y_2) \, dy_1dy_2dt.
\end{equation}

Let $\mathcal{S}$ be the set of smooth and cylindrical random variables of the form
\[
F=g\lp W(h_{1}),\ldots,W(h_{N})\rp,
\]
where $N\ge 1$, $g\in C_{b}^{\infty}(\mathbb{R}^{N})$ and $h_{1},\ldots,h_{N}\in\ch_{T}$. For every $\ell\in\ch_{T}$, the partial Malliavin derivative of $F$ in the direction of $\ell$ is defined for $F\in \mathcal{S}$ as the random variable
\begin{equation}\label{derivative}
D_{\ell}F=\sum_{i}^{N}\frac{\partial g}{\partial x_{i}}\lp W(h_{1}),\ldots,W(h_{N})\rp
\lla h_{i}, \ell\rra_{\ch_{T}} .
\end{equation}
Relation \eqref{derivative} can also be seen as an equation for $\langle DF, \, \ell \rangle_{\ch_{T}}$, where $DF$ is now a $\mathcal{H}_{T}$-valued random variable.
We can iterate this procedure to define higher order derivatives $D^{k}_{\ell_{1}\cdots\ell_{k}}F$, which produces a $\mathcal{H}^{\otimes k}$-valued random variable. For any $p\ge 1$ and integer $k\ge 1$, we define the Sobolev space $\mathbb{D}^{k,p}$ as the closure of $\mathcal{S}$ with respect to the norm
\begin{equation}\label{Sobolev norm}
\|F\|_{k,p}^{p}=\mathbb{E}[|F|^{p}]
+\sum_{i=1}^{k} \mathbb{E}\left[\|D^{i}F\|^{p}_{\mathcal{H}^{\otimes l}} \right].
\end{equation}
If $V$ is Hilbert space, $\mathbb{D}^{k,p}(V)$ denotes the corresponding Sobolev space of $V$-valued random variables.

The existence of a density for $u(t,x)$, the mild solution of \eqref{eq:SPDE}, is obtained by appling the following criterion borrowed from \cite[Theorem 2.1.2]{nualart-book}.

\begin{proposition}\label{prop:malliavin>0}
  Let $F$ be a real-valued random variable in $\mathbb{D}^{1,p}$ for some $p>1$, such that $\|DF\|_{\H_T}>0$ with probability one. Then the law of $F$ is absolutely continuous with respect to the Lebesgue measure in $\mathbb{R}$.
\end{proposition}

When proceeding by approximations, we will rely on a technical result summarized below (see \cite[Lemma 1.2.3]{nualart-book}) in order to probe Malliavin differentiability.

\begin{proposition}\label{prop:Malliavin-approx}
  Let $\{F_\lambda: \lambda>0\}$ be a sequence of random variables such that
  \begin{equation}
  F_\lambda \in \mathbb{D}^{1,2} \ \text{ for all $\lambda>0$,}\quad
  \lim_{\lambda \to 0} \E \lc| F_\lambda - F|^2\rc= 0, \quad
  \sup_\lambda \E\lc\| D F_\lambda\|_{\H_T}^2 \rc< +\infty .
  \end{equation}
    Then $F \in \mathbb{D}^{1,2}$ and the sequence $\{D F_\lambda: \lambda>0\}$ converges weakly to $DF$ in $L^2(\Omega: \H_T)$ as $\lambda \to 0$.
\end{proposition}

We now state a differentiation rule for stochastic integrals that will be invoked to differentiate solutions to stochastic \textsc{pde}s.

\begin{proposition}[{Section 1.3.1 of \cite{nualart-book}}]\label{prop:mall-diff-for-stoch-intg}
Let $X$ be an adapted random field in $\D^{1,2}(\ch_{T})$, and define the It\^o stochastic integral
\begin{equation*}
\cj_{T}(X)
=
\int_{0}^{T}\int_{\R^{d}} X(s,x) \, W(ds \, dy) .
\end{equation*}
Then $\cj_{T}(X)$ is an element of $\D^{1,2}$, and for any $h\in\ch_{T}$ we have
\begin{equation*}
D_{h} \cj_{T}(X)
=
\lla X, \, h \rra_{\ch_{T}}
+
\int_{0}^{T}\int_{\R^{d}} D_{h}X(s,x) \, W(ds \, dy) .
\end{equation*}
\end{proposition}

The next result gives an easy to check condition that guarantees that a stochastic convolution with the fundamental solution to the  heat equation satisfies the assumptions of Proposition \ref{prop:mall-diff-for-stoch-intg}.

\begin{proposition} \label{prop:mall-diff-stoch-conv}
  Let $X$ be an adapted random field such that for any $(t,x) \in [0,T]\times \mathbb{R}^d$, we have $X(t,x) \in \D^{1,2}$. In addition, we assume that
  \begin{equation}\label{a801}
    \sup_{t \in [0,T]}\sup_{x \in \mathbb{R}^d} \E\lc |X(t,x)|^2 \rc<+\infty,
  \end{equation}
  and
  \begin{equation}\label{a802}
    \sup_{t \in [0,T]}\sup_{x \in \mathbb{R}^d} \E\lc| DX(t,x)|_{\H_T}^2 \rc < +\infty.
  \end{equation}
  Let $G$ be the fundamental solution of the heat equation. Define the stochastic convolution
  \begin{equation} \label{eq:I-integral-def}
    \ci(t,x) = \ci^X(t,x) = \int_0^t \int_{\mathbb{R}^d} G(t-s,x-y)X(s,y)W(ds \, dy).
  \end{equation}
  Then for any $(t,x) \in [0,T]\times \mathbb{R}^d$,  $\ci(t,x)$ is an element of $\D^{1,2}$, and for any $h \in \ch_{T}$ we have
  \begin{equation} \label{eq:mall-deriv-stoch-int}
     D_{h} \ci(t,x)
     =
     \lla G(t-\cdot,x-\cdot)X(\cdot,\cdot), \, h \rra_{\ch_{T}}
     +
     \int_{0}^{t}\int_{\R^{d}} G(t-s,x-y)D_{h}X(s,y) \, W(ds \, dy) .
  \end{equation}
  In the above expression $G(t-r,x-y)$ is defined to be $0$ is $r>t$.
\end{proposition}

\begin{proof}
  To apply Proposition \ref{prop:mall-diff-for-stoch-intg}, we need to first verify that for fixed $(t,x) \in [0,T]\times \mathbb{R}^d$, the integrand $(s,y) \mapsto G(t-s,x-y)X(s,y)$ is an element of $\mathbb{D}^{1,2}(\H_T)$. Fortunately, a straightforward consequence of the fact that $G(\cdot,\cdot)$ and $\Lambda(\cdot)$ are positive is that
  \begin{align} \label{eq:sigma-u-in-L2-H}
    &\E|G(t-\cdot,x-\cdot)X(\cdot,\cdot)|_{\H_T}^2 \nonumber\\
    &=\E\int_0^t \int_{\mathbb{R}^d} \int_{\mathbb{R}^d}G(t-s,x-y_1)G(t-s,x-y_2)X(s,y_1)X(s,y_2)\Lambda(y_1-y_2)dy_1dy_2ds \nonumber\\
    &\leq C \left( \sup_{s \in [0,t]} \sup_{y \in \mathbb{R}^d} \E|X(s,y)|^2 \right)   Q_{\laa}(t) \,
  \end{align}
  where we have set
  \begin{equation}\label{a81}
  Q_{\laa}(t)
  =
  \int_0^t \int_{\mathbb{R}^d}\int_{\mathbb{R}^d} G(t-s,x-y_1)G(t-s,x-y_2)\Lambda(y_1-y_2)dy_1dy_2ds \, .
  \end{equation}
  Along the same lines, we also have
  \begin{multline*}
  \E|G(t-\cdot,x-\cdot)DX(\cdot,\cdot)|_{\H_T}^2
    =\E\int_0^t \int_{\mathbb{R}^d}\int_{\mathbb{R}^d} G(t-s,x-y_1)G(t-s,x-y_2) \\
\times    \left<DX(s,y_1), DX(s,y_2) \right>_{\ch_{T}}\Lambda(y_1-y_2)dy_1dy_2ds,
  \end{multline*}
  and therefore
  \begin{equation}\label{a82}
   \E|G(t-\cdot,x-\cdot)DX(\cdot,\cdot)|_{\H_T}^2
   \leq
   C \left(\sup_{s \in [0,t]} \sup_{y \in \mathbb{R}^d} \E|DX(s,y)|_{\H_T}^2 \right)
    Q_{\laa}(t) .
  \end{equation}
    Now
  taking Fourier transforms as in \cite{dalang}, Assumption \ref{assum:dalang} guarantees that
  \begin{equation}\label{a821}
    Q_{\laa}(t) = \int_0^t \int_{\mathbb{R}^d} e^{-(t-s)|\xi|^2} \mu(d\xi) <+\infty.
  \end{equation}
Plugging this relation into \eqref{eq:sigma-u-in-L2-H} and \eqref{a82}, then taking hypothesis~\eqref{a801}-\eqref{a802} into account, our claim is a direct consequence of
  Proposition \ref{prop:mall-diff-for-stoch-intg}.
\end{proof}

\section{Malliavin differentiability of $\mathcal{M}$}\label{sec:mall-diff-M}
In Section \ref{sec:approx-scheme} we gave the iteration scheme allowing to solve~\eqref{eq:mild-solution}. We shall now follow the very same scheme in order to show Malliavin differentiability, and we start by analyzing the mapping $\cm$ defined by~\eqref{eq:M-def}. Namely
we showed in \cite{s-2021} that  $\mathcal{M}$ is a Lipschitz continuous map on the weighted spaces $\cac_{\theta,x_0}([0,T]\times \mathbb{R}^d)$. In this section we prove the following proposition about the Malliavin differentiability of $\mathcal{M}(z)$.

\begin{proposition} \label{prop:M-Malliavin}
Let $\theta \in \left( 0, \frac{2}{\nu} \right)$ where $\nu$ is from \eqref{eq:f-exp-bound}, and pick a generic $x_0 \in \mathbb{R}^d$.
Denote by $L^2(\Omega:\cac_{\theta,x_0}([0,T]\times\mathbb{R}^d))$ the set of $\cac_{\theta,x_0}([0,T]\times\mathbb{R}^d))$-valued square integrable random variables.
Consider
$z \in L^2(\Omega:\cac_{\theta,x_0}([0,T]\times\mathbb{R}^d))$ which has the properties that $z(t,x)$ is Malliavin differentiable for all $(t,x) \in [0,T]\times \mathbb{R}^d$ and
  \begin{equation} \label{oo}
\E  \left|\sup_{t \in [0,T]} \sup_{x \in \mathbb{R}^d} \frac{|Dz(t,x)|_{\H_T}}{1 + |x-x_0|^\theta} \right|^2<  +\infty.
  \end{equation}
  Let $\cm(z)$ be given as in Definition~\ref{def:M} and assume Assumptions \ref{assum:sig}--\ref{assum:f} for $\sigma$ and $b$.
  Then $\mathcal{M}(z)(t,x)$ is also Malliavin differentiable  for all $(t,x) \in [0,T]\times \mathbb{R}^d$ and almost surely we have
  \begin{equation} \label{eq:M-Malliavin-Lipschitz}
    \sup_{t \in [0,T]} \sup_{x \in \mathbb{R}^d} \frac{|D[\mathcal{M}(z)(t,x)]|_{\H_T}}{1 + |x-x_0|^\theta} \leq K \sup_{t \in [0,T]} \sup_{x \in \mathbb{R}^d} \frac{|Dz(t,x)|_{\H_T}}{1 + |x-x_0|^\theta} ,
%    \text{ with probability one.}
  \end{equation}
  where $K=K(T,\theta,\kappa)$ is also the Lipschitz constant of $\mathcal{M}$ in~\eqref{eq:M-Lip}, which does not depend on $x_0 \in \mathbb{R}^d$.
\end{proposition}

We prove Proposition~\ref{prop:M-Malliavin} in several steps. First, we prove this in the simpler case where $f$ is globally Lipschitz continuous.
\begin{lemma} \label{lem:M-Malliavin-if-Lipschitz}
Let $\theta \in \left( 0, \frac{2}{\nu} \right)$ where $\nu$ is from \eqref{eq:f-exp-bound}.  Let $x_0 \in \mathbb{R}^d$.
  Assume that $f:\mathbb{R} \to \mathbb{R}$ is differentiable and globally Lipschitz continuous and that $\sup_u f'(u) \leq \kappa$. Then if $z \in L^2(\Omega:\cac_{\theta,x_0}([0,T]\times\mathbb{R}^d))$  has the properties that $z(t,x)$ is Malliavin differentiable for all $(t,x) \in [0,T]\times \mathbb{R}^d$ and
  \begin{equation}\label{a83}
\E \Bigg|\sup_{t \in [0,T]} \sup_{x \in \mathbb{R}^d} \frac{|Dz(t,x)|_{\H_T}}{1 + |x-x_0|^\theta} \Bigg|^2< +\infty,
  \end{equation}
  then $\mathcal{M}(z)$ is also Malliavin differentiable and with probability one we have
  \begin{equation}\label{a9}
    \sup_{t \in [0,T]} \sup_{x \in \mathbb{R}^d} \frac{|D[\mathcal{M}(z)(t,x)]|_{\H_T}}{1 + |x-x_0|^\theta} \leq K \sup_{t \in [0,T]} \sup_{x \in \mathbb{R}^d} \frac{|Dz(t,x)|_{\H_T}}{1 + |x-x_0|^\theta} \, ,
  \end{equation}
  where $K=K(T,\theta,\kappa)$ is also the Lipschitz constant of $\mathcal{M}$. Notice that $K$ depends on $\kappa$, the upper bound of $f'(u)$, but does not depend on the lower bound of $f'(u)$.
\end{lemma}

\begin{proof}
  We define a sequence of functions $\{m_n: n\geq 1\}$ by Picard iterations by
  \begin{align*}
   & m_1(t,x) = z(t,x)\\
    &m_{n+1}(t,x) = \int_0^t \int_{\mathbb{R}^d} G(t-s,x-y)f(m_n(s,y))dyds + z(t,x).
  \end{align*}
  Standard arguments based on the Lipschitz continuity of $f$ show that $m_n$ converge to $m:=\mathcal{M}(z)$ in the $L^2(\cac_{\theta,x_0}([0,T]\times \mathbb{R}^d))$ topology.
  Furthermore, $m_1(t,x)$ is Malliavin differentiable by assumption. Then by induction and using the fact that integration and Malliavin differentiability commute,
  $m_n(t,x)$ is Malliavin differentiable  for all $n\ge 1$ and
  \begin{equation} \label{eq:Dm-mild-Lipschitz}
    Dm_{n+1}(t,x) = \int_0^t \int_{\mathbb{R}^d} G(t-s,x-y) f'(m_n(s,y))Dm_n(s,y)dyds + Dz(t,x).
  \end{equation}
  Notice that in order to get \eqref{eq:Dm-mild-Lipschitz}, we imposed the additional assumption that $f'$ is bounded. Let $L:=\sup_{u \in \mathbb{R}} |f'(u)|$. We also set
  \begin{equation}\label{a10}
    \Phi_n(t): = \sup_{x \in \mathbb{R}^d} \sup_{s \in [0, t]}\frac{|D m_n(s,x)|_{\H_T}}{1 + |x-x_0|^\theta}.
  \end{equation}
  We assumed in the statement of the lemma that $\Phi_1(t)$ is finite with probability one, since $m_{1}=z$ and $z$ satisfies~\eqref{a83}. From \eqref{eq:Dm-mild-Lipschitz} we can see that for any $n\ge 1$ and $(t,x)\in[0,T]\times\R^{d}$ we have
  \begin{equation}\label{b1}
    |Dm_{n+1}(t,x)|_{\H_T} \leq \int_0^t  \int_{\mathbb{R}^d} G(t-s,x-y) L |Dm_n(s,y)|_{\H_T}dyds + |Dz(t,x)|_{\H_T}.
  \end{equation}
  Next, we observe that $|Dm_n(s,y)| \leq \Phi_n(s)(1 + |y|^\theta)$ by the definition \eqref{a10} of $\Phi_{n}$. Furthermore, due to the fact that $G$ is a Gaussian kernel, we have
  \begin{equation}\label{b2}
    \int_{\mathbb{R}^d} G(t-s,x-y) (1 + |y-x_0|^\theta) dy\leq C ( 1 + (t-s)^{\frac{ \theta}{2}} + |x-x_0|^\theta).
  \end{equation}
  Plugging \eqref{b2} into \eqref{b1}, there exists $C_T>0$ such that for any $t \in [0,T]$
  \begin{equation}
    \Phi_{n+1}(t) \leq C_T \int_0^t \Phi_n(s) ds + \sup_{r \in [0,T]} \sup_{x \in \mathbb{R}^d} \frac{|Dz(r,x)|_{\ch_{T}}}{1 + |x-x_0|^\theta}.
  \end{equation}
  Hence it is easily seen by induction that $\Phi_n$ satisfies the following inequality, uniformly in $n$,
  \begin{equation}
    \sup_{t \in [0,T]} \Phi_{n}(t) \leq e^{C_T T}  \sup_{t \in [0,T]} \sup_{x \in \mathbb{R}^d} \frac{|Dz(t,x)|}{1 + |x-x_0|^\theta} ,
  \end{equation}
  with probability one.
In particular, if we fix $(t,x)\in [0,T]\times\mathbb{R}^d$, we get
  \begin{equation}
    \sup_n \E |Dm_n(t,x)|^2 \leq c_{T,x}\,  \E \left|\sup_{r \in [0,T]} \sup_{x \in \mathbb{R}^d} \frac{|Dz(r,x)|}{1 + |x-x_0|^\theta} \right|^2<+\infty.
  \end{equation}
  Therefore by Proposition \ref{prop:Malliavin-approx} and the fact that $\E|m_n(t,x) - m(t,x)|^2 \to 0,$ the random variable $m(t,x)$ is Malliavin differentiable and $Dm_n(t,x)$ converges weakly to $Dm(t,x)$ in $L^2(\Omega; \H_T)$ as $n \to \infty$. Taking limits in~\eqref{eq:Dm-mild-Lipschitz} thanks to a standard procedure, $Dm(t,x)$ must be the  solution to
  \begin{equation} \label{eq:Dm-mild-Lipschitz-2}
    Dm(t,x) = \int_0^t G(t-s,x-y) f'(m(s,y))Dm(s,y)dyds + Dz(t,x),
  \end{equation}
  where we recall that \eqref{eq:Dm-mild-Lipschitz-2} admits a unique solution if $\sup_u |f'(u)|=L<\infty$.

Now that we have shown that $Dm(t,x)$ exists and solves \eqref{eq:Dm-mild-Lipschitz-2}, we argue that we can improve the bound on $|Dm(t,x)|_{\H_T}$ so that it depends only on the upper bound $\kappa:= \sup_u f'(u) $ and not on the Lipschitz constant $L=\sup_u |f'(u)|$.
To this aim, by the linearity of \eqref{eq:Dm-mild-Lipschitz-2}, for any $h \in \H_T$,
\begin{equation} \label{eq:Dm-h-mild}
  D_hm(t,x) = \int_0^t G(t-s,x-y) f'(m(s,y))D_hm(t,x)dyds + D_hz(t,x).
\end{equation}
Recall that $m = \mathcal{M}(z)$ is the unique solution to \eqref{eq:M-def} and define the function $\varphi:[0,T]\times\mathbb{R}^d\times \mathbb{R}\to \mathbb{R}$ by
\begin{equation}
  \varphi(t,x,V) = f'(m(t,x))V.
\end{equation}
Because $\kappa:= \sup_u f'(u)<+\infty$, the following inequality holds for any $t \in [0,T]$, $x \in \mathbb{R}^d$, and $V_1> V_2 $:
\begin{equation}
  \varphi(t,x,V_1) - \varphi(t,x,V_2) \leq \kappa (V_1 - V_2).
\end{equation}
The growth condition \eqref{eq:varphi-growth-restriction} is fulfilled for $\vp$ because $m \in \cac_{\theta,x_0}$ and $f$ is globally Lipschitz continuous by assumption. Hence for $L = \sup_u|f'(u)|$,
\begin{equation}
  |\varphi(t,x,V)| \leq L  |V| .
\end{equation}
Therefore, $D_h m$ satisfies the assumptions of Theorem \ref{thm:M-Lip-inhomo}. Moreover, it is readily checked that for $D_h z \equiv 0$, equation~\eqref{eq:Dm-h-mild} admits $D_h m(t,x) =0$ as a unique solution. Hence there exists $K=K(T,\theta,\kappa)$ depending only on $T, \theta,$ and $\kappa$, but not $L$, such that
  \begin{equation}
    \sup_{t \in [0,T]} \sup_{x \in \mathbb{R}^d} \frac{|D_h m(t,x) -0 |}{1 + |x-x_0|^\theta} \leq K \sup_{t \in [0,T]} \sup_{x \in \mathbb{R}^d} \frac{|D_h z(t,x)-0|}{1 + |x-x_0|^\theta} \, ,
  \end{equation}
  with probability one.
  In particular, for any $t \in [0,T], x \in \mathbb{R}^d$ and $|h|_{\H_T}=1$
  \begin{equation}
    \frac{|D_h m(t,x)|}{1 + |x-x_0|^\theta} \leq K \sup_{t \in [0,T]} \sup_{x \in \mathbb{R}^d} \frac{|Dz(t,x)|_{\H_T}}{1 + |x-x_0|^\theta} \, ,
  \end{equation}
  with probability one.
  Taking the supremum over $|h|_{\H_T}=1$, $t \in [0,T]$ and $x \in \mathbb{R}^d$,
  \begin{equation}
    \sup_{t \in [0,T]} \sup_{x \in \mathbb{R}^d} \frac{|Dm(t,x)|_{\H_T}}{1 + |x|^\theta} \leq K \sup_{t \in [0,T]} \sup_{x \in \mathbb{R}^d} \frac{|Dz(t,x)|_{\H_T}}{1 + |x-x_0|^\theta} .
  \end{equation}
  This proves our claim \eqref{a9} under the assumption that $f$ is globally Lipschitz continuous.
\end{proof}

Now we start a limiting procedure in order to prove Proposition \ref{prop:M-Malliavin}. Namely let $f$ be any force satisfying Assumption \ref{assum:f} and let $f_\lambda$ be the Yosida approximation defined by \eqref{eq:f-Yosida}.
By Lemma \ref{lem:M-Malliavin-if-Lipschitz}, for each $\lambda>0$, because each $f_\lambda$ is Lipschitz continuous, there exists a unique $m_\lambda$ solving (see also item~\ref{it:from-Z-n-to-u-n} in Section~\ref{sec:approx-scheme}):
\begin{equation} \label{eq:m-lambda}
  m_\lambda(t,x) = \int_0^t \int_{\mathbb{R}^d} G(t-s,x-y)f_\lambda(m_\lambda(s,y))dyds + z(t,x).
\end{equation}
Due to the fact that $f_{\la}$ satisfies \ref{item:half-lip-f-lambda} in Lemma \ref{lem:f-Yosida-properties} uniformly in $\la$  (for a given $\ka>0$), it follows from Theorem~\ref{thm:M-Lip-inhomo} (see also \cite[Corollary 5.5]{s-2021}) that $m_{\la}$ is such that
\begin{equation} \label{eq:m-lambda-bound}
  \sup_{t \in [0,T]} \sup_{x \in \mathbb{R}^d} \frac{|m_\lambda(t,x)|}{1 + |x-x_0|^\theta}
  \leq
  K \lp 1 + \sup_{t \in [0,T]} \sup_{x \in \mathbb{R}^d} \frac{|z(t,x)|}{1 + |x-x_0|^\theta} \rp ,
\end{equation}
where the constant $K$ depends only on $T, \theta$ and $\kappa$.
Moreover, Lemma \ref{lem:M-Malliavin-if-Lipschitz} guarantees that $m_\lambda(t,x)$ is Malliavin differentiable. According to~\eqref{eq:Dm-mild-Lipschitz-2},  the Malliavin derivative satisfies
\begin{equation}
  D_h m_\lambda(t,x) = \int_0^t \int_{\mathbb{R}^d} G(t-s,x-y)f'_\lambda(m_\lambda(s,y))D_h m_\lambda(s,y)dyds
  + D_h z(t,x).
\end{equation}
In this context, relation \eqref{a9} can be read as
\begin{equation} \label{eq:Dm-lambda-bound}
  \sup_{t \in [0,T]} \sup_{x \in \mathbb{R}^d} \frac{|D m_\lambda(t,x)|_{\H_T}}{1 + |x|^\theta} \leq K \sup_{t \in [0,T]} \sup_{x \in \mathbb{R}^d} \frac{|Dz(t,x)|_{\H_T}}{1 + |x|^\theta} \, ,
\end{equation}
and the constant $K$ is like in \eqref{eq:m-lambda-bound}. Notice again from Lemma \ref{lem:f-Yosida-properties} that all of the $f_\lambda$ have the same half-Lipschitz constant $\kappa$.

In the following lemma we improve on the approximation results in~\cite{s-2021}, and show that $m_\lambda(t,x)$ converges in $L^2(\Omega)$.
\begin{lemma} \label{lem:Yosida-conv}
  Let $m_\lambda$ be the Yosida approximations defined in \eqref{eq:m-lambda} and assume that $z$ satisfies the assumptions of Proposition \ref{prop:M-Malliavin}. Then for any fixed $(t,x) \in [0,T]\times\mathbb{R}^d$, we have
  \begin{equation}
    \lim_{\lambda \to 0} \E |m_\lambda(t,x) - m(t,x)|^2 = 0 \, ,
  \end{equation}
  where $m(t,x)$ is the unique solution to
  \begin{equation}
    m(t,x) = \int_0^t \int_{\mathbb{R}^d} G(t-s,x-y) f(m(s,y))dyds + z(t,x).
  \end{equation}
\end{lemma}

\begin{proof}
  First, we prove that $m_\lambda(t,x)$ converges almost surely to $m(t,x)$. Let $v_\lambda(t,x) = m_\lambda(t,x) - z(t,x)$. Then
  \begin{equation}
    v_\lambda(t,x) = \int_0^t \int_{\mathbb{R}^d} G(t-s,x-y)f_\lambda(m_\lambda(s,y))dyds
  \end{equation}
  By the growth rate assumption \eqref{eq:f-exp-bound} and the bound \eqref{eq:m-lambda-bound},
  \begin{equation} \label{eq:f-u-lambda-bound}
    |f_\lambda(m_\lambda(t,x))| \leq C ( 1 + |x-x_0|^\theta) \exp( C (1 + |x-x_0|^{\nu \theta})) \, ,
  \end{equation}
  where $C = C(\omega)$ is some random constant that is independent of $\lambda$. If $\nu \theta<2$, \eqref{eq:f-u-lambda-bound} along with standard properties of the heat kernel $G$ prove that for almost any fixed $\omega \in \Omega$
the family $\{(t,x)\mapsto v_\lambda(t,x): \lambda\in (0,1)\}$
 is equibounded and equicontinuous for $(t,x)$ in any compact subset of $[0,T]\times \mathbb{R}^d$. The Arzela-Ascoli theorem guarantees that there exists a subsequence $\lambda_n \to 0$ such that $v_{\lambda_n}(t,x)$ converges to a limit $v^*(t,x)$, uniformly on compact subsets of $(t,x) \in [0,T]\times\mathbb{R}^d$. Because $m_\lambda = v_\lambda + z$, invoking the bound \eqref{eq:f-u-lambda-bound}, properties (ii) and (iv) in Lemma \ref{lem:f-Yosida-properties} plus some standard dominated convergence arguments, we get that this limit solves
 \begin{equation}
   v^*(t,x) = \int_0^t \int_{\mathbb{R}^d} G(t-s,x-y) f(v^*(s,y) + z(s,y))dyds .
 \end{equation}
 Thus if we define $m^*(t,x) := v^*(t,x) + z(t,x)$, then $m^*$ solves
 \begin{equation} \label{eq:m-star}
   m^*(t,x) = \int_0^t \int_{\mathbb{R}^d} G(t-s,x-y)f(m^*(s,y))dyds + z(t,x).
 \end{equation}
 The solution to \eqref{eq:m-star} is unique by Proposition \ref{prop:M-Lip}. Therefore, the same Arzela-Ascoli argument proves that every subsequence of $m_\lambda$ has a further subsequence that converges to $m^*$ and therefore $\lim_{\lambda \to 0} m_\lambda(t,x) = m^*(t,x)$ with probability one.

 Finally, the $L^2(\Omega)$ convergence of $m_\lambda(t,x)$ to $m^*(t,x)$ is a consequence of the almost sure bound \eqref{eq:m-lambda-bound}, the assumption that $z \in L^2(\Omega: C_{\theta}([0,T]\times\mathbb{R}^d))$, and the dominated convergence theorem.
\end{proof}

After this series of preliminary Lemmas, we are now ready to prove the main result of this section.

\begin{proof}[Proof of Proposition \ref{prop:M-Malliavin}]
  Let $z$ satisfy the assumptions of Proposition \ref{prop:M-Malliavin} and let $m_\lambda$ be defined in \eqref{eq:m-lambda}. By \eqref{eq:Dm-lambda-bound} and the fact that $z$ satisfies~\eqref{oo}, for any fixed $(t,x) \in [0,T]\times \mathbb{R}^d$,
  \begin{equation}
    \sup_{\lambda \in (0,1)}\E|Dm_\lambda(t,x)|_{\ch_{T}}^2 < +\infty.
  \end{equation}
  Lemma \ref{lem:Yosida-conv} proves that
  \begin{equation}
    \lim_{\lambda \to 0} \E \left| m_\lambda(t,x) - m(t,x)\right|^2 = 0.
  \end{equation}
  The assumptions of Proposition \ref{prop:Malliavin-approx} are therefore satisfied. The limit $m(t,x)$ is Malliavin differentiable and verifies~\eqref{eq:M-Malliavin-Lipschitz}.
  \end{proof}

\begin{remark}
Applying Proposition \ref{prop:Malliavin-approx} as above, we get the weak convergence of $Dm_\lambda(t,x)$ to $Dm(t,x)$. Combining this with  the almost sure bounds \eqref{eq:m-lambda-bound}, \eqref{eq:Dm-lambda-bound} and the exponential growth assumption  \eqref{eq:f-exp-bound}, similarly to what we did for \eqref{eq:m-star}, guarantees that $Dm(t,x)$ solves
  \begin{equation} \label{eq:Dm-mild}
    Dm(t,x) = \int_0^t \int_{\mathbb{R}^d} G(t-s,x-y) f'(m(s,y))Dm(s,y)dyds  + Dz(t,x).
  \end{equation}
  where $m = \mathcal{M}(z)$.

\end{remark}

\section{Malliavin differentiability of the mild solution} \label{sec:mall-diff-I}

As mentioned at the beginning of Section~\ref{sec:mall-diff-M}, our analysis of the Malliavin differentiability of $u$ (solution to~\eqref{eq:mild-solution}) follows the scheme outlined in Section~\ref{sec:approx-scheme}. In this section we focus our attention on the stochastic convolution $Z$ given in the iteration step \eqref{eq:Z-Picard1}.
We prove Malliavin differentiability via an approximation procedure and Proposition~\ref{prop:Malliavin-approx}. Some preliminary results are presented in Section~\ref{sec:moment-bounds} and the Malliavin differentiability is achieved in Section~\ref{sec:mall-diff}.

\subsection{Some moment bounds}\label{sec:moment-bounds}
  In \cite[Theorem 2.1]{ss-2002}, it is proved that certain stochastic convolutions $\ci(t,x) = \int_0^t G(t-s,x-y)\varphi(s,y)W(ds \, dy)$ are H\"older continuous in $t$ and $x$ as long as $\varphi$ is a real-valued adapted random field satisfying $\sup_{t \in [0,T]} \sup_{x \in \mathbb{R}^d} \E|\vp(t,x)|^p<\infty$ for sufficiently large $p>1$. The following lemma claims that the same result is true when $\varphi$ is Hilbert-space valued. Due to our Malliavin calculus considerations, in the current paper we are mostly interested in the case where the integrand is $\H_T$-valued.
\begin{lemma} \label{lem:Holder-cont-Hilbert}
  Suppose that $\varphi$ is an adapted $\H_T$-valued  adapted random field defined on $[0,T]\times\R^{d}$. We assume
  \begin{equation}
    \sup_{t \in [0,T]} \sup_{x \in \mathbb{R}^d} \E |\varphi(t,x)|_{\H_T}^p<\infty \, ,
  \end{equation}
  for sufficiently large $p>1$.
  As in \eqref{eq:I-integral-def}, let
\begin{equation}\label{eq:def-Phi}
    \ci(t,x) = \ci^\varphi(t,x) = \int_0^t \int_{\mathbb{R}^d} G(t-s,x-y) \varphi(s,y)W(ds \, dy) \, ,
\end{equation}
and recall that $\eta$ is introduced in the strong Dalang condition~\eqref{eq:strong-dalang}.
  Then for any $0<\ga<\al<\eta/2$  and $p \geq 2$, there exist constants $C_{\alpha,\gamma}>0$  such that for any $x, x_1, x_2 \in \mathbb{R}^d$ and $t,t_1,t_2 \in [0,T]$,
  \begin{align}
    &\E|\ci(t,x)|^p_{\H_T} \leq
    C_{\alpha,\gamma} T^{p\alpha} \sup_{s \in [0,T]} \sup_{y \in \mathbb{R}^d} \E|\varphi(s,y)|^p_{\H_T}\\
    &\E|\ci(t,x_1) - \ci(t,x_2)|^p_{\H_T}
    \leq C_{\alpha,\gamma} |x_1 - x_2|^{2\gamma p}
  \,  T^{p(\alpha-\ga)} \sup_{s \in [0,T] } \sup_{y \in \mathbb{R}^d} \E|\varphi(s,y)|^p_{\H_T}\\
    &\E|\ci(t_1,x) - \ci(t_2,x)|^p_{\H_T}
    \leq C_{\alpha,\gamma} |t_1 - t_2|^{\gamma p}
    \, T^{p(\alpha-\ga)} \sup_{s \in [0,T]} \sup_{y \in \mathbb{R}^d} \E|\varphi(s,y)|^p_{\H_T}.
  \end{align}
\end{lemma}
\begin{proof}
  The only difference between the proof in \cite{ss-2002} and this lemma is that $\varphi$ and $\ci$ are Hilbert-space-valued. Theorem 4.36 of \cite{dpz-book} guarantees that the BDG inequality holds for Hilbert-space valued stochastic integrals identically to how it holds for real-valued random variables.
\end{proof}
In the sequel we will need to bound stochastic integrals in the norms related to Definition~\ref{def:weighted-cont-space}. Towards this aim, the following lemma generalizes Theorem 4.3 of \cite{s-2021} to the Hilbert-space setting of Lemma \ref{lem:Holder-cont-Hilbert}. Its proof is omitted for sake of conciseness.
\begin{lemma} \label{lem:weighted-space-Hilbert}
  Let $T>0$. Consider $p> \frac{2(d+1)}{\eta}$ where $d$ is the spatial dimension and $\eta \in (0,1)$ is still the parameter from~\eqref{eq:strong-dalang}. For any $\theta> \frac{d+1}{p-(d+1)}$, there exists $C_{T,p,\theta}>0$ such that for all adapted random fields $\varphi: [0,T] \times \mathbb{R}^d \times \Omega \to \H_T$ verifying
  \begin{equation}
    \sup_{t \in [0,T]}\sup_{x \in \mathbb{R}^d} \E|\varphi(t,x)|_{\H_T}^p < +\infty,
  \end{equation}
  the stochastic integral $\ci = \ci^\varphi$ defined by \eqref{eq:def-Phi}
  satisfies
  \begin{equation} \label{eq:weighted-Lp-Hilbert}
    \sup_{x_0 \in \mathbb{R}^d} \E \left|\sup_{t \in [0,T]} \sup_{x \in \mathbb{R}^d} \frac{|\ci(t,x)|_{\H_T}}{1 + |x-x_0|^\theta} \right|^p
    \leq C_{T,p,\theta} \sup_{t \in [0,T]} \sup_{x_0 \in \mathbb{R}^d} \E |\varphi(t,x_0)|^p_{\H_T} \, .
  \end{equation}
  Moreover, the constant $C_{T,p,\theta}$ defined in  \eqref{eq:weighted-Lp-Hilbert} satisfies $\lim_{T \to 0} C_{T,p,\theta}=0$.
\end{lemma}

%\begin{remark}
%\hb{
%Here we might want to include something justifying the use of the norms
%\begin{equation*}
%\sup_{t \in [0,T]} \sup_{x \in \mathbb{R}^d} \frac{|\ci(t,x)|_{\H_T}}{1 + |x-x_0|^\theta} \, ,
%\end{equation*}
%vs norms in $C_{\te}$
%}
%\end{remark}

Next we need to bound weighted norms of $\ch_{T}$-valued processes, in a suitable way for our stochastic convolutions. A deterministic type result in this direction is presented below.

\begin{lemma} \label{lem:H-norm-conv}
  Assume that $X$ is a function indexed by $[0,T]\times\R^{d}$ such that there exists $\theta \geq 0$ and $x_{0}$ satisfying
  \begin{equation}\label{b3}
    M:= \sup_{t \in [0,T]} \sup_{x \in \mathbb{R}^d} \frac{|X(t,x)|}{1 + |x-x_0|^\theta}< +\infty.
  \end{equation}
  Then there exists $C=C(T,\theta)>0$ such that we have
  \begin{equation} \label{eq:H-norm-conv}
    \sup_{t \in [0,T]} \sup_{x \in \mathbb{R}^d}\frac{\left|G(t-\cdot, x-\cdot) X(\cdot,\cdot) \mathbbm{1}_{[0,t]} \right|_{\H_T}^2}{(1 + |x-x_0|^\theta)^2} \leq C M^2 Q_\Lambda(T) \,
  \end{equation}
  where $Q_\Lambda$ is defined in \eqref{a821}.
  If $X$ is a random field, then \eqref{eq:H-norm-conv} holds with probability one.
\end{lemma}

\begin{proof}
  Assume that there exists $\theta \geq 0$ and $T>0$ such that~\eqref{b3} holds.
 the definition of the $\H_T$ norm \eqref{eq:H_T}, we write
  \begin{align}\label{b4}
    &\left|G(t-\cdot, x-\cdot) X(\cdot,\cdot) \mathbbm{1}_{[0,t]} \right|_{\H_T}^2 \nonumber \\
    &= \int_0^t \int_{\mathbb{R}^d} \int_{\mathbb{R}^d} G(t-s,x-y_1)G(t-s,x-y_2) X(s,y_1) X(s,y_2)
    \Lambda(y_1-y_2)dy_1dy_2.
  \end{align}
 Now owing to the definition of $M$ and the positivity of $G$ and $\Lambda$, an upper bound for the right hand side of~\eqref{b4} is
  \begin{multline}\label{b5}
    \left|G(t-\cdot, x-\cdot) X(\cdot,\cdot) \mathbbm{1}_{[0,t]} \right|_{\H_T}^2
    \leq M^2 \int_0^t \int_{\mathbb{R}^d} \int_{\mathbb{R}^d} G(t-s,x-y_1)G(t-s,x-y_2) \\
     \times(1 + |y_1-x_0|^\theta)(1 + |y_2-x_0|^\theta)\Lambda(y_1-y_2)dy_1dy_2.
  \end{multline}
  Moreover, the quantities $1 + |y_1-x_0|^\theta$ and $1 + |y_2-x_0|^\theta$ above satisfy
  \begin{equation}\label{b6}
    1 + |y_i-x_0|^\theta \leq C(1 + |y_i - x|^\theta + |x-x_0|^\theta ) = C (1 + |x-x_0|^\theta + |y_i - x|^\theta).
  \end{equation}
In addition, invoking the elementary relation $\sup_{z\in\R}|z|^{\te}\exp(-a\, x^{2})\le c_{a}<\infty$, valid for any arbitrary constant $a>0$,
  it is easy to verify that there exists $C = C(\theta)$ such that
  \begin{equation}\label{b7}
    |y|^{\theta} G(t,y) = t^\frac{\theta}{2} \left| \frac{y}{\sqrt{t}}\right|^\theta G(t,y) \leq C t^{\frac{\theta}{2}}  G(t/2,y).
  \end{equation}
  Plugging \eqref{b6} and \eqref{b7} into \eqref{b5}, we thus get
  \begin{align} \label{eq:H_T-squared}
    &\left|G(t-\cdot, x-\cdot) X(\cdot,\cdot) \mathbbm{1}_{[0,t]} \right|_{\H_T}^2 \\
    &\leq CM^2 \left(1 + |x-x_0|^{\theta} \right)^{2}
    \int_0^t \int_{\mathbb{R}^d} \int_{\mathbb{R}^d} G(t-s,x-y_1)G(t-s,x-y_2)
     \Lambda(y_1-y_2)dy_1dy_2 \nonumber\\
    &\qquad +CM^2 \int_0^t \int_{\mathbb{R}^d} \int_{\mathbb{R}^d} (t-s)^{\theta}G((t-s)/2,x-y_1)G((t-s)/2,x-y_2) \Lambda(y_1-y_2)dy_1dy_2 . \nonumber
  \end{align}
The right hand side above can now be expressed in terms of Fourier transforms and a change of variable $s:=t-s$ in the time integral. We end up with
\begin{multline}
\left|G(t-\cdot, x-\cdot) X(\cdot,\cdot) \mathbbm{1}_{[0,t]} \right|_{\H_T}^2  \\
\le
C M^2 (1+|x-x_0|^\theta)^2  \int_0^t \int_{\mathbb{R}^d} e^{-s|\xi|^2}\mu(d\xi)ds
+ CM^2\int_0^t \int_{\mathbb{R}^d} s^{\theta} e^{-\frac{s|\xi|^2}{2}} \mu(d\xi)ds \, .
\end{multline}
With \eqref{a821} in mind, it is thus readily checked that
\begin{equation*}
\left|G(t-\cdot, x-\cdot) X(\cdot,\cdot) \mathbbm{1}_{[0,t]} \right|_{\H_T}^2
\le
CM^2(1 + |x-x_0|^\theta)^2 Q_\Lambda(t) +  CM^2t^\theta {Q}_\Lambda(t) \, ,
\end{equation*}
from which our claim \eqref{eq:H-norm-conv} is easily deduced.
\end{proof}

\subsection{Malliavin differentiability}\label{sec:mall-diff}
With the above preliminary results in hand, let us state our Malliavin differentiability result for the mild solution to the \textsc{spde} \eqref{eq:mild-solution}.
\begin{theorem} \label{thm:u-mall-diff}
Recall that our coefficients $b,\si$ satisfy Assumptions~\ref{assum:sig}--\ref{assum:f}.
Let $u$ be the unique solution to \eqref{eq:mild-solution} and let $Z$ be the stochastic convolution defined in \eqref{eq:Z-limit}. For any fixed $(t,x)\in[0,T]\times\R^{d}$,  both $u(t,x)$ and $Z(t,x)$ are Malliavin differentiable. For any time horizon $T>0$ and power $p>1$, the Malliavin derivatives satisfy
  \begin{equation}\label{eq:Z-mall-derivative}
    \sup_{x_0 \in \mathbb{R}^d} \E \left|\sup_{t \in [0,T]} \sup_{x \in \mathbb{R}^d} \frac{|DZ(t,x)|_{\H_T}}{1 + |x-x_0|^\theta} \right|^p< +\infty
  \end{equation}
 and
  \begin{equation}\label{eq:u-mall-derivative}
    \sup_{x_0 \in \mathbb{R}^d} \E \left|\sup_{t \in [0,T]} \sup_{x \in \mathbb{R}^d} \frac{|Du(t,x)|_{\H_T}|}{1 + |x-x_0|^\theta} \right|^p < + \infty
  \end{equation}
\end{theorem}

\begin{proof}
  We will prove this theorem via Proposition \ref{prop:Malliavin-approx}, using the approximating sequences $\{Z_{n}; \, n \ge 0\}$ and $\{u_{n}; \, n \ge 0\}$ that are  respectively defined by \eqref{eq:Z-Picard1} and~\eqref{eq:u-Picard}. We proceed by induction to prove a uniform bound on $Du_n$ and $DZ_n$ in an appropriate topology. For $n\ge 0$, we thus suppose that both $u_{n}(t,x)$ and $Z_n(t,x)$ are elements of the space $\mathbb{D}^{1,p}$ given by \eqref{Sobolev norm} with $p>1$ and satisfy
    \begin{align}
        \label{eq:DZ_n-bound}
    &\sup_{x_0 \in \mathbb{R}^d} \E \left|\sup_{t \in [0,T]} \sup_{x \in \mathbb{R}^d} \frac{|DZ_n(t,x)|_{\H_T}}{1 + |x-x_0|^\theta} \right|^p< +\infty\\
        \label{eq:Du_n-bound}
    &\sup_{x_0 \in \mathbb{R}^d} \E \left|\sup_{t \in [0,T]} \sup_{x \in \mathbb{R}^d} \frac{|Du_n(t,x)|_{\H_T}}{1 + |x-x_0|^\theta} \right|^p< +\infty.
  \end{align}
  The inductive assumption \eqref{eq:DZ_n-bound}--\eqref{eq:Du_n-bound}  is trivially true for the initial $U_0$ and $Z_0$, defined in \eqref{eq:U-Picard0}--\eqref{eq:Z-Picard0} because they are non-random.
  We shall now propagate this assumption.  We first consider $Z_{n+1}$ defined by~\eqref{eq:Z-Picard1}.
Namely we are assuming \eqref{eq:Du_n-bound} holds for $u_n$ and we wish to apply Proposition \ref{prop:mall-diff-stoch-conv} to prove that $Z_{n+1}$ defined by \eqref{eq:Z-Picard1} is Malliavin differentiable. To this aim, Proposition \ref{prop:mall-diff-stoch-conv} will be applicable if $\si(u_n)$ satisfies \eqref{a801}--\eqref{a802}. We proceed to check those assumptions below.

  From \eqref{eq:u_n-conv}, for each $n \in \mathbb{N}$,
  \begin{equation}
    \sup_{x_0 \in \mathbb{R}^d} \E \left|\sup_{t \in [0,T]} \sup_{x \in \mathbb{R}^d} \frac{|u_n(t,x)|}{1 + |x-x_0|^\theta} \right|^p< +\infty.
  \end{equation}
  By setting $x = x_{0}$ in the above expression and moving the time supremum outside of the expectation,
  \begin{eqnarray}
    \sup_{x_0 \in \mathbb{R}^d} \sup_{t \in[0,T]} \E |u_{n}(t,x_0)|^p
    &\leq& \sup_{x_0 \in \mathbb{R}^d} \E\left|\sup_{t \in [0,T]} |u_n(t,x_0)| \right|^p \nonumber\\
    &\leq&
    \sup_{x_0 \in \mathbb{R}^d} \E \left| \sup_{t \in [0,T]} \sup_{x \in \mathbb{R}^d}
    \frac{ |u_n(t,x)|}{1 + |x-x_0|^\theta} \right|^p < +\infty.
  \end{eqnarray}
  Therefore, $u_n$ satisfies \eqref{a801}. The same reasoning shows that \eqref{eq:Du_n-bound} implies that $Du_n$ satisfies relation~\eqref{a802}. In summary,
  \begin{equation*}
\sup_{t \in [0,T]}\sup_{x \in \mathbb{R}^d} \E|u_n(t,x)|^2 <+\infty,
\quad\text{and}\quad
\sup_{t \in [0,T]}\sup_{x \in \mathbb{R}^d} \E|Du_n(t,x)|_{\H_T}^2 <+\infty.
\end{equation*}
Since $\si$ verifies Assumptions~\ref{assum:sig}, we also have
 \begin{equation*}
\sup_{t \in [0,T]}\sup_{x \in \mathbb{R}^d} \E|\si\lp u_n(t,x)\rp |^2 <+\infty,
\quad\text{and}\quad
\sup_{t \in [0,T]}\sup_{x \in \mathbb{R}^d} \E|D\si\lp u_n(t,x)\rp|_{\H_T}^2 <+\infty.
\end{equation*}
  Now observe that the stochastic integral $Z_{n+1}$ defined in \eqref{eq:Z-Picard1} can be written as $\ci^{\sigma(u_n)}$ using the notation of \eqref{eq:I-integral-def}. Therefore, Proposition \ref{prop:mall-diff-stoch-conv} asserts that $Z_{n+1}(t,x)$ is an element of $\mathbb{D}^{1,2}$ and
\begin{align} \label{eq:DZ_n-formula}
D_{h}Z_{n+1}(t,x)
    = &\left<G(t-\cdot,x-\cdot)\sigma(u_n(\cdot,\cdot)),h \right>_{\H_T} \nonumber\\
&+ \int_0^t \int_{\mathbb{R}^d} G(t-s,x-y) \sigma'(u_n(s,y)) D_{h}u_n(s,y) \, W(ds \, dy) \, .
\end{align}

Having shown that $Z_{n+1}(t,x)\in\mathbb{D}^{1,2}$ for all $(t,x)\in[0,T]\times\R^{d}$, we proceed to prove~\eqref{eq:Z-mall-derivative} and~\eqref{eq:u-mall-derivative}. To this aim, observe that relation \eqref{eq:DZ_n-formula} can be written as
  \begin{equation}
    D_h Z_{n+1}(t,x) = \left<G(t-\cdot,x-\cdot)\sigma(u_n), h \right>_{\H_T} + \ci^{\sigma'(u_n)D_hu_n}(t,x) \, ,
  \end{equation}
  where we invoked our notation \eqref{eq:I-integral-def} again. We can bound the first term on the right-hand-side thanks to \eqref{eq:H-norm-conv} and the second term on the right-hand side (remember that $\si'$ is uniformly bounded according to Assumption~\ref{assum:sig}) with \eqref{eq:weighted-Lp-Hilbert} to get
  \begin{align} \label{eq:DZ_n-Lp}
    &\sup_{x_0 \in \mathbb{R}^d} \E \left|\sup_{t \in [0,T]} \sup_{x \in \mathbb{R}^d} \frac{|DZ_{n+1}(t,x)|_{\H_T}}{1 + |x-x_0|^\theta} \right|^p\\
    &\leq C (Q_\Lambda(T))^{\frac{p}{2}} \sup_{x_0 \in \mathbb{R}^d}\E\left| \sup_{t \in [0,T]} \sup_{x \in \mathbb{R}^d} \frac{|\sigma(u_n(t,x))|}{1 + |x-x_0|^\theta}\right|^p + C_{T,p,\theta}\sup_{x_0 \in \mathbb{R}^d} \E \left|\sup_{t \in [0,T]} \sup_{x \in \mathbb{R}^d} \frac{|Du_{n}(t,x) |_{\H_T}}{1 + |x-x_0|^\theta} \right|^p .
    \nonumber
  \end{align}
  In addition, invoking Lemma \ref{lem:weighted-space-Hilbert}, we have
  \begin{equation}  \label{eq:C-small}
    \lim_{T \downarrow 0} C_{T,p,\theta}=0,
  \end{equation}
  and we also recall that $Q_\Lambda(T)$ is defined in \eqref{a821}.

  We are now ready to propagate relation \eqref{eq:Du_n-bound} for $Du_n$ on a small time interval. Namely, because $u_{n+1} = \mathcal{M}(U_0 + Z_{n+1})$, relations \eqref{eq:M-Malliavin-Lipschitz} and \eqref{eq:DZ_n-Lp} guarantee that
  \begin{align} \label{eq:Du_n-Lp}
    &\sup_{x_0 \in \mathbb{R}^d} \E \left|\sup_{t \in [0,T]} \sup_{x \in \mathbb{R}^d} \frac{|Du_{n+1}(t,x)|_{\H_T}}{1 + |x-x_0|^\theta} \right|^p
   % \nonumber\\&
   \leq C\sup_{x_0 \in \mathbb{R}^d}\E \left|\sup_{t \in [0,T]} \sup_{x \in \mathbb{R}^d} \frac{|DU_0(t,x)|_{\H_T}}{1 + |x-x_0|^\theta} \right|^p \\
     &+C (Q_\Lambda(T))^{\frac{p}{2}} \sup_{x_0 \in \mathbb{R}^d}\E\left| \sup_{t \in [0,T]} \sup_{x \in \mathbb{R}^d} \frac{|\sigma(u_n(t,x))|}{1 + |x-x_0|^\theta}\right|^p + C_{T,p,\theta}\sup_{x_0 \in \mathbb{R}^d} \E \left|\sup_{t \in [0,T]} \sup_{x \in \mathbb{R}^d} \frac{|Du_{n}(t,x) |_{\H_T}}{1 + |x-x_0|^\theta} \right|^p .
     \nonumber
  \end{align}
  In addition, if the initial data $u_0$ is non-random, then $DU_0(t,x) \equiv 0$. We  choose $T_0>0$ small enough so that the coefficient $C_{T_0,p,\theta}<1$. Then by induction, invoking~\eqref{eq:Du_n-Lp} we get
  \begin{multline}\label{b8}
    \sup_n \sup_{x_0 \in \mathbb{R}^d} \E \left|\sup_{t \in [0,T_0]} \sup_{x \in \mathbb{R}^d} \frac{|Du_{n}(t,x)|_{\H_T}}{1 + |x-x_0|^\theta} \right|^p \\
    \leq
    \frac{C(Q_\Lambda(T))^{\frac{p}{2}}}{1 - C_{T_0,p,\theta}} \sup_n   \sup_{x_0 \in \mathbb{R}^d}\E\left| \sup_{t \in [0,T]} \sup_{x \in \mathbb{R}^d} \frac{|\sigma(u_n(t,x))|}{1 + |x-x_0|^\theta}\right|^p .
  \end{multline}
  The right hand side of the above expression is finite because of \eqref{eq:u_n-conv} and the assumption that $\sigma$ is globally Lipschitz continuous. Therefore,
  \begin{equation} \label{eq:Du_n-uniform}
    \sup_n \sup_{x_0 \in \mathbb{R}^d} \E \left|\sup_{t \in [0,T_0]} \sup_{x \in \mathbb{R}^d} \frac{|Du_{n}(t,x)|_{\H_T}}{1 + |x-x_0|^\theta} \right|^p <+\infty.
  \end{equation}
  Then \eqref{eq:DZ_n-Lp}, \eqref{eq:u_n-conv}, and \eqref{eq:Du_n-uniform} imply that
  \begin{equation} \label{eq:DZ_n-uniform}
    \sup_n \sup_{x_0 \in \mathbb{R}^d} \E \left|\sup_{t \in [0,T_0]} \sup_{x \in \mathbb{R}^d} \frac{|DZ_{n}(t,x)|_{\H_T}}{1 + |x-x_0|^\theta} \right|^p <+\infty.
  \end{equation}
  In particular, for any fixed $(t,x) \in [0,T_0]\times \mathbb{R}^d$, we have $\sup_n \E|Du_n(t,x)|_{\H_T}^2<+\infty$ and $\sup_n\E|DZ_n(t,x)|_{\H_T}^2<+\infty$. Approximations \eqref{eq:Z_n-conv}--\eqref{eq:u_n-conv} and Proposition \ref{prop:Malliavin-approx} guarantee that $u(t,x)$ and $Z(t,x)$ are Malliavin differentiable for $t \in [0,T_0]$ and $x \in \mathbb{R}^d$, where $T_{0}$ is the small parameter chosen above. Furthermore, \eqref{eq:Z-mall-derivative}--\eqref{eq:u-mall-derivative} hold for $T<T_0$ by Fatou's lemma

  We can extend this result to arbitrary time horizons $T>T_0$ by taking advantage of the self-similarity of the process $W$. Namely for any $n \in \mathbb{N}$ and $t>0$,
  \begin{align*}
    u_{n+1}(T_0 + t,x) = &\int_{\mathbb{R}^d} G(t,x-y) u_{n+1}(T_0,y)dy + \int_0^t \int_{\mathbb{R}^d} G(t-s,x-y) f(u_n(T_0 + s,y))dyds \nonumber\\&+ \int_0^t \int_{\mathbb{R}^d}G(t-s,x-y) \sigma(u_n(T_0+s))W(dy,(T_0+ds))
  \end{align*}
  This translated process has the property that
  \begin{equation*}
    u_{n+1}(T_0 + t,x) = \mathcal{M} \left(U_{T_0,n+1} + \tilde{Z}_{n+1} \right)(t,x)
  \end{equation*}
  where we have set
  \begin{eqnarray*}
  U_{T_0,n+1}(t,x) &:=& \int_{\mathbb{R}^d} G(t,x-y) u_{n+1}(T_0,y)dy \\
  \tilde{Z}_{n+1}(t,x) &:=& \int_0^t \int_{\mathbb{R}^d}G(t-s,x-y) \sigma(u_n(T_0+s))W(dy,(T_0+ds))
  \end{eqnarray*}
    The $p$-th moment bounds \eqref{eq:DZ_n-Lp}--\eqref{eq:Du_n-Lp} continue to hold with $U_0$ replaced by $U_{T_0,n+1}$ and $Z_n$ replaced by $\tilde{Z_n}$. The constants $C_{T,p,\theta}$ and $Q_\Lambda(T)$ in these expressions do not change (they only depend on the law of $W$, which is invariant by time shift) and we can use the same value of $T_0$ as above. The only novelty is that $DU_0\equiv 0$, while $DU_{T_0,n+1}$ is not zero. But we have a uniform bound on
  \begin{equation*}
    \sup_n \sup_{x_0 \in \mathbb{R}^d} \E \left|\sup_{t \in [0,T_0]} \sup_{x \in \mathbb{R}^d} \frac{|DU_{T_0,n+1}(t,x)|_{\H_T}}{1 + |x-x_0|^\theta} \right|^p \, ,
  \end{equation*}
  as a consequence of \eqref{eq:Du_n-uniform} and properties of the heat kernel $G$. Therefore, the same arguments that implied \eqref{eq:Du_n-uniform}--\eqref{eq:DZ_n-uniform} imply
  \begin{equation} \label{eq:Du_n-uniform-2}
    \sup_n \sup_{x_0 \in \mathbb{R}^d} \E \left|\sup_{t \in [T_0,2T_0]} \sup_{x \in \mathbb{R}^d} \frac{|Du_{n}(t,x)|_{\H_T}}{1 + |x-x_0|^\theta} \right|^p <+\infty.
  \end{equation}
  It follows immediately from \eqref{eq:DZ_n-Lp} that
  \begin{equation} \label{eq:DZ_n-uniform-2}
    \sup_n \sup_{x_0 \in \mathbb{R}^d} \E \left|\sup_{t \in [T_0,2T_0]} \sup_{x \in \mathbb{R}^d} \frac{|DZ_{n}(t,x)|_{\H_T}}{1 + |x-x_0|^\theta} \right|^p <+\infty.
  \end{equation}
  Proposition \ref{prop:Malliavin-approx} guarantees that $u(t,x)$ and $Z(t,x)$ are Malliavin differentiable for $t \in [T_0,2T_0]$.
  Repetition of this argument proves the Malliavin differentiability of $u(t,x)$ and $Z(t,x)$ for  $t \in [2T_0, 3T_0], [3T_0,4T_0]$, and so on.
\end{proof}

We close this section by identifying an integral equation satisfied by directional Malliavin derivatives of $u$
\begin{corollary}
  Recall that $\H_T$ is the Hilbert space with inner product given by \eqref{eq:H_T} and assume the same hypotheses as in Theorem \ref{thm:u-mall-diff}. For any $T>0$ and $h \in \H_T$, and $(t,x) \in [0,T]\times \mathbb{R}^d$, the Malliavin derivative of $u(t,x)$ solves the integral equation
  \begin{align} \label{eq:Du-mild}
    D_h u(t,x) = &\left<G(t-\cdot,x-\cdot)\sigma(u(\cdot,\cdot)), h \right>_{\H_T} \nonumber\\
    &+ \int_0^t \int_{\mathbb{R}^d} G(t-s,x-y) f'(u(s,y))D_hu(s,y)dyds \nonumber\\
    &+ \int_0^t \int_{\mathbb{R}^d} G(t-s,x-y)\sigma'(u(s,y))D_h u(s,y)W(ds \, dy).
  \end{align}
\end{corollary}

\begin{proof}
  The integral equation \eqref{eq:Du-mild} is easily obtained by combining \eqref{eq:Dm-mild} with \eqref{eq:DZ_n-uniform} and Proposition \ref{prop:mall-diff-for-stoch-intg}.
\end{proof}

\section{Positivity of the Malliavin derivative} \label{sec:positivity}
In this section we prove that the solution $u$ to \eqref{eq:mild-solution} has the property that $u(t,x)$ admits a density for $t>0$, $x \in \mathbb{R}^d$ by analyzing the quantity $\|Du(t,x)\|_{\ch_T}$. Our main theorem can be expressed as below.

\begin{theorem}\label{thm:dens-u}
Let $u$ solution to \eqref{eq:mild-solution}, where we assume that the coefficients $b,\si$ satisfy Assumptions~\ref{assum:sig}--\ref{assum:f}. We also suppose that the Gaussian noise $\dw$ verifies Hypothesis~\ref{assum:dalang}. Then for $t\in(0,T]\times\R^{d}$, the random variable $u(t,x)$ admits a density with respect to Lebesgue measure.
\end{theorem}

In order to prove Theorem \ref{thm:dens-u}, we start in Subsection \ref{SS:positivity-strategy} by explaining our strategy, and we perform most of our computations in Subsections \ref{SS:lower-and-upper-bounds}--\ref{SS:Du-lower}.

\subsection{Strategy} \label{SS:positivity-strategy}

According to Proposition \ref{prop:malliavin>0}, the existence of a density is established if we can prove that $u(t,x) \in \mathbb{D}^{1,p}$ for some $p\geq 1$ and if the Malliavin derivative has strictly positive $\H_T$ norm. Moreover, Theorem \ref{thm:u-mall-diff} asserts that $u(t,x)$ belongs to the Sobolev space $\mathbb{D}^{1,p}$ for all $p >1$. It remains to prove that for every $t>0$, $x \in \mathbb{R}^d$,
\begin{equation} \label{eq:Du-pos-prob-1}
  \Pro \left(\|Du(t,x)\|_{\H_T}>0 \right)=1.
\end{equation}
For convenience we will further reduce \eqref{eq:Du-pos-prob-1} to a statement about directional derivatives. Namely condition \eqref{eq:Du-pos-prob-1} is implied if %we can find a deterministic function $h \in \H_T$ such that
\begin{equation}\label{c1}
  \Pro \left(\textnormal{There exists } h \in \H_T \textnormal{ such that } D_h u (t,x)>0 \right)=1.
\end{equation}

We now elaborate on \eqref{c1}. Recall the integral equation \eqref{eq:Du-mild} for the directional Malliavin derivative of $u(t,x)$. Namely for any $h \in \H_T$, $t \in [0,T]$, and $x \in \mathbb{R}^d$, we have obtained that
\begin{equation} \label{eq:Dhu-decomposed}
  D_h u(t,x) = \left< \phi_{t,x},h \right>_{\H_T} + A_h(t,x) + B_h(t,x),
\end{equation}
where the function $\phi_{t,x}$ is defined on $[0,T]\times \mathbb{R}^d$ by
\begin{equation} \label{eq:phi-t-x}
  \phi_{t,x}(r,z) = G(t-r,x-z)\sigma(u(r,z))\mathbbm{1}_{[0,t]}(r) \, ,
\end{equation}
and where the terms $A_h$ and $B_h$ are given by
\begin{eqnarray} \label{eq:A_h}
  A_h(t,x) &=& \int_0^t \int_{\mathbb{R}^d} G(t-s,x-y)f'(u(s,y))D_hu(s,y)dyds \\
 \label{eq:B_h}
  B_h(t,x) &=& \int_0^t \int_{\mathbb{R}^d} G(t-s,x-y) \sigma'(u(s,y))D_h u(s,y) W(ds \, dy).
\end{eqnarray}
Furthermore, one can also express the directional derivative of $u(t,x)$ as
\begin{equation} \label{eq:D_h=L}
  D_h u(t,x) = \mathcal{L}\left( \left< \phi_{\cdot,\cdot},h \right>_{\H_T} + B_h(\cdot,\cdot) \right)(t,x) \, ,
\end{equation}
where $\mathcal{L}$ is the functional mapping defined in \eqref{eq:varphi-mild} with $\varphi(t,x,V) := f'(u(s,y))V$. This point of view allows us to apply Theorem \ref{thm:M-Lip-inhomo} to get useful upper bounds on $|D_hu(t,x)|$.

With \eqref{c1} in mind, for fixed $t_0 \in [0,T]$, $x_0 \in \mathbb{R}$, and $\delta\in (0,t_0)$ we chose a specific test function $h_{t_0,x_0,\delta} \in \H_T$. Specifically, let $h_{t_0,x_0,\delta}$ be  defined by
\begin{equation} \label{eq:h-delta-def-1}
    h_{t_0,x_0,\delta}(r,z): = G(t_0-r,x_0-z)\mathbbm{1}_{[t_0-\delta,t_0]}(r).
\end{equation}
Because $h_{t_0,x_0,\delta}(r,z)$ is only nonzero when $r \in [t_0-\delta,t_0]$ and $\phi_{t,x}(r,z)$ defined in \eqref{eq:phi-t-x} is only non-zero when $r \in [0,t]$, the inner product $\left<\phi_{t,x},h_{t_0,x_0,\delta}\right>_{\H_T}$ is zero when $t \in [0,t_0-\delta)$. Then because the equations \eqref{eq:Dhu-decomposed}--\eqref{eq:B_h} are linear with respect to $D_{h_{t_0,x_0,\delta}} u(t,x)$, it follows that $D_{h_{t_0,x_0,\delta}} u(t,x)=0$ for all $t \in [0,t_0-\delta]$. This means that the integral terms $A_{h_{t_0,x_0,\delta}}$ and $B_{h_{t_0,x_0,\delta}}$ in \eqref{eq:A_h} and \eqref{eq:B_h} can be written as integrals starting at $t_0-\delta$. Specifically, for any $t \in [t_0-\delta,t_0]$, we have
\begin{eqnarray} \label{eq:A_h-delta-interval}
  A_{h_{t_0,x_0,\delta}}(t,x)
  &=&
 \int_{t_0-\delta}^t \int_{\mathbb{R}^d} G(t-s,x-y)f'(u(s,y))D_{h_{t_0,x_0,\delta}}u(s,y)dyds
\\
\label{eq:B_h-delta-interval}
  B_{h_{t_0,x_0,\delta}}(t,x)
  &=&
  \int_{t_0-\delta}^t \int_{\mathbb{R}^d} G(t-s,x-y)\sigma'(u(s,y))D_{h_{t_0,x_0,\delta}}u(s,y)W(ds \, dy) \, .
\end{eqnarray}
Moreover, for $t \in [0,t_0-\delta]$ both $A_{h_{t_0,x_0,\delta}}(t,x)$ and $B_{h_{t_0,x_0,\delta}}(t,x)$ are vanishing.

With those preliminary considerations in hand, using the decomposition \eqref{eq:Dhu-decomposed}, we will achieve the desired property~\eqref{c1}. That is we will be able to show that with probability one there exists $\delta>0$ such that $D_{h_{t_0,x_0,\delta}}u(t_0,x_0) >0$. This is accomplished by proving that  $\left<\phi_{t_0,x_0}, h_{t_0,x_0,\delta}\right>_{\ch_{T}}$ is larger than $A_{h_{t_0,x_0,\delta}}(t_0,x_0) + B_{h_{t_0,x_0,\delta}}(t_0,x_0)$ when $\delta$ is small, guaranteeing that \eqref{eq:Dhu-decomposed} is positive.
The main ideas are the following.
\begin{enumerate}[wide,  labelindent=0pt, label=\emph{(\roman*)}]
\setlength\itemsep{.02in}

\item\label{c2}
Invoking the non degeneracy of the kernel $G$ and of the coefficient $\si$, we shall see that
the inner product $\left<\phi_{t_0,x_0}, h_{t_0,x_0,\delta}\right>_{\ch_{T}}$ has a lower bound that is proportional to $Q_\Lambda(\delta)$ defined in \eqref{a821}.

\item\label{c3}
Using \eqref{eq:D_h=L} along with the boundedness of $\sigma'$, we will prove that an upper bound on the $L^p(\Omega:{\cac_{\theta,x_0}([0,T]\times\mathbb{R}^d)})$ norm of
$D_{h_{t_0,x_0,\delta}}u(t,x)$ is also proportional to $Q_\Lambda(\delta)$.

\item\label{c4}
Finally, the fact that $A_{h_{t_0,x_0,\delta}}(t_0,x_0)$, and $B_{h_{t_0,x_0,\delta}}(t_0,x_0)$ in \eqref{eq:A_h-delta-interval}--\eqref{eq:B_h-delta-interval} are integrals whose integrands have $L^p(\Omega)$ norms proportional to $Q_\Lambda(\delta)$, but whose interval of integration are of size $\delta$, we can show that with probability one there exists a subsequence $\delta_k \downarrow 0$ such that
\begin{equation}
  \lim_{k \to \infty} \frac{|A_{h_{t_0,x_0,\delta_k}}(t_0,x_0)| + |B_{h_{t_0,x_0,\delta_k}}(t_0,x_0)|}{Q_\Lambda(\delta_k)} = 0,
  \quad  \text{ with probability one.}
\end{equation}
\end{enumerate}

\noindent
Putting together the 3 items above, one concludes that with probability one, there exists a small enough (random) $\delta_k = \delta_k(\omega)$ such that
\begin{equation}
  D_{h_{t_0,x_0,\delta_k}}u(t_0,x_0) = \left< \phi_{t,x},h_{t_0,x_0,\delta_k} \right>_{\H_T} + A_{h_{t_0,x_0,\delta_k}}(t_0,x_0) + B_{h_{t_0,x_0,\delta_k}}(t_0,x_0)>0.
\end{equation}
This will prove \eqref{c1} and therefore the existence of a density. The remainder of the section is devoted to detail the 3 steps above.

\subsection{Lower and upper bounds on $\left<\phi_{t,x}, h_{t_0,x_0,\delta}\right>_{\ch_{T}}$} \label{SS:lower-and-upper-bounds}
In this section we give details about item \ref{c2} above. Our findings are summarized in the following lemma.
\begin{lemma} \label{lem:phi-h}
  Let $t_0>0, x_0 \in \mathbb{R}^d$. For $\delta \in [0,t_0]$ we recall the definition of $h_{t_0,x_0,\delta}$ in~\eqref{eq:h-delta-def-1}:
  \begin{equation} \label{eq:h-delta-def}
    h_{t_0,x_0,\delta}(r,z): = G(t_0-r,x_0-z)\mathbbm{1}_{[t_0-\delta,t_0]}(r).
  \end{equation}
Also recall that $\phi_{t,x}$ is defined by \eqref{eq:phi-t-x} and $Q_\Lambda$ is introduced in \eqref{a81}--\eqref{a821}. Then we have the following lower and upper bounds for $\left<\phi_{t,x}, h_{t_0,x_0,\delta}\right>_{\ch_{T}}$.
  For any $T>0,t_0 \in [0,T], x_0 \in \mathbb{R}^d$ and $\delta \in [0,t_0]$
  \begin{equation} \label{eq:phi-h-product-lower}
    \left<\phi_{t_0,x_0}, h_{t_0,x_0,\delta}\right>_{\ch_T} \geq \alpha \, Q_\Lambda(\delta) \, ,
    \quad\text{ with probability one} \, 
  \end{equation}
  where $\alpha>0$ is from \eqref{eq:sigma-non-degen}.
In addition, there exists $C>0$ such that for any $T>0, t_0\in [0,T], x_0 \in \mathbb{R}^d, y_0 \in \mathbb{R}^d$, and  $\delta \in [0,t_0]$,
  \begin{equation} \label{eq:phi-h-product-upper}
    \sup_{t \in [t_0-\delta,t_0]} \sup_{x \in \mathbb{R}^d}\frac{|\left<\phi_{t,x}, h_{t_0,x_0,\delta} \right>_{\ch_{T}}|}{1 + |x-y_0|^\theta} \leq C Q_{\Lambda}(\delta) \left( 1 + \sup_{s \in [t_0-\delta, t_0]} \sup_{y \in \mathbb{R}^d} \frac{|u(s,y)|}{1 + |y-y_0|^\theta} \right).
  \end{equation}
  Furthermore, if $t \in [0, t_0-\delta], x_0 \in \mathbb{R}^d$, and $x \in \mathbb{R}^d$,
  \begin{equation} \label{eq:phi-h-product-0}
    \left<\phi_{t,x}, h_{t_0,x_0,\delta}\right>_{\ch_{T}} = 0.
  \end{equation}
\end{lemma}

\begin{proof}
  Recall the expressions \eqref{eq:phi-t-x} and \eqref{eq:h-delta-def} for $\phi_{t,x}$ and $h_{t_0,x_0,\delta}$ as well as the definition~\eqref{eq:H_T} of the inner product in $\H_T$. Then it is readily checked that
  \begin{align} \label{eq:phi-h-product-integral}
    &\left< \phi_{t,x}, h_{t_0,x_0,\delta} \right>_{\H_T} \nonumber\\
    &= \int_{t_0-\delta}^{t} \int_{\mathbb{R}^d} \int_{\mathbb{R}^d} G(t-s,x-y_1)\sigma(u(s,y_1))G(t-s,x_{0}-y_2)\Lambda(y_1-y_2)dy_1dy_2ds,
  \end{align}
  for all $t \in [t_0-\delta,t_0]$.
  We can now prove \eqref{eq:phi-h-product-lower} in the following way: owing to the fact that $\sigma(u(s,y_1)) > \alpha$ (see Assumption \ref{assum:sig}) and thanks to the positivity of $G$ and $\Lambda$ we obtain
  \begin{align}
    &\left<\phi_{t_0,x_0}, h_{t_0,x_0,\delta} \right>_{\H_T} \nonumber\\
    & \geq \alpha \int_{t_0-\delta}^{t_0} \int_{\mathbb{R}^d} \int_{\mathbb{R}^d} G(t_0-s, x_0-y_1) G(t_0-s,x_0-y_2)\Lambda(y_1-y_2)dy_1dy_2ds.
  \end{align}
  The last term in the above expression is $Q_\Lambda(\delta)$, which was defined in \eqref{a81}. We have thus obtained
  \begin{equation}
    \left<\phi_{t_0,x_0}, h_{t_0,x_0,\delta}  \right>_{H_T} \geq \alpha Q_\Lambda(\delta),
  \end{equation}
  that is \eqref{eq:phi-h-product-lower} holds true.

  The proof of \eqref{eq:phi-h-product-upper} is a consequence of Lemma \ref{lem:H-norm-conv}.
  Namely start from the expression~\eqref{eq:phi-h-product-integral} for the inner product
  $\langle \phi_{t_0,x_0}, h_{t_0,x_0,\delta} \rangle_{\H_{T}}$. Then notice that in the right hand side of~\eqref{eq:phi-h-product-integral},  the time integrand goes from $t_0-\delta$ to $t$. By the Cauchy Schwarz inequality, we get
 \begin{equation}\label{eq:cauchy-sch-in-initial}
\left\langle\phi_{t_0,x_0}, h_{t_0,x_0,\delta}  \right\rangle_{H_T}
\le
(R_1(t,x))^{\frac{1}{2}}(R_2(t))^{\frac{1}{2}} \, ,
\end{equation}
where we have set
\begin{multline}\label{c5}
R_1(t,x)
=
\int_{t_0-\delta}^{t} \int_{\mathbb{R}^d}\int_{\mathbb{R}^d}G(t-s,x-y_1)G(t-s,x-y_2) \\
\times \sigma(u(s,y_1))\sigma(u(s,y_2))\Lambda(y_1-y_2)dy_1dy_2
\end{multline}
and
\begin{equation}\label{c6}
R_2(t)
=
\int_{t_0-\delta}^{t} \int_{\mathbb{R}^d}\int_{\mathbb{R}^d}G(t-s,x_0-y_1)G(t-s,x_0-y_2)\Lambda(y_1-y_2)dy_1dy_2 \, .
\end{equation}
 The  term $R_2(t)$ in \eqref{c6} is easily bounded above by
 \begin{equation}\label{c7}
 R_2(t) \leq Q_\Lambda(\delta) \, .
 \end{equation}
 In fact $R_2(t)$ would be exactly $Q_\Lambda(\delta)$ if we had $t = t_0$. The term $R_1(t,x)$ defined by~\eqref{c5} can be bounded using a time translated modification of Lemma \ref{lem:H-norm-conv} with $T=\delta$ (the length of the time interval $[t_0-\delta,t_{0}]$) and $X(t,x) = \sigma(u(t,x)$. Specifically, we get
  \begin{equation*}
    \sup_{t \in [t_0-\delta,t_0]}\sup_{x \in \mathbb{R}^d} \frac{R_1(t,x)}{(1 + |x-x_0|^\theta)^2}
    \leq C Q_\Lambda(\delta) \left( \sup_{t \in [t_0-\delta,t_0]} \sup_{x \in \mathbb{R}^d} \frac{|\sigma(u(t,x))|^2}{(1 + |x-x_0|^\theta)^2} \right).
  \end{equation*}
  Since $\si$ has linear growth, we thus end up with
  \begin{equation}\label{c8}
  \sup_{t \in [t_0-\delta,t_0]}\sup_{x \in \mathbb{R}^d} \frac{R_1(t,x)}{(1 + |x-x_0|^\theta)^2}
    \leq C Q_\Lambda(\delta) \left( 1 +\sup_{t \in [t_0-\delta,t_0]}  \sup_{x \in \mathbb{R}^d}  \frac{|u(t,x)|^2}{(1 + |x-x_0|^\theta)^2} \right).
  \end{equation}
  Then \eqref{eq:phi-h-product-upper} follows by plugging \eqref{c7} and \eqref{c8} into~\eqref{eq:cauchy-sch-in-initial}.

  Finally, if $t \in [0, t_0-\delta]$ then the supports  of $\phi_{t,x}$ and $h_{t_0,x_0,\delta}$ are disjoint and therefore relation~\eqref{eq:phi-h-product-0} follows.
\end{proof}

\subsection{Upper bounds on moments of the derivative of $\mathbf{u(t,x)}$} \label{SS:Du-upper}
We have seen in \eqref{eq:u-mall-derivative} that the Malliavin derivative of $u(t,x)$ is bounded in the $L^{p}$-sense for all $p>1$. Equation~\eqref{b8} can even be seen as a quantitative bound on this derivative. In the current section we push forward this analysis to derive an upper bound for the moments of the weighted supremum norms
\begin{equation}
  \sup_{y_0 \in \mathbb{R}^d}\E \sup_{t \in [t_0-\delta,t_0]}\sup_{x \in \mathbb{R}^d}\left| \frac{|D_{h_{t_0,x_0,\delta}}u(t,x)|}{1 + |x-y_0|^\theta} \right|^p \, ,
\end{equation}
and to show that these are proportional to $(Q_\Lambda(\delta))^{p}$. As mentioned in our strategy Section~\ref{SS:positivity-strategy}, this is the contents of item \ref{c3} above.
We will use \eqref{eq:D_h=L} along with Theorem \ref{thm:M-Lip-inhomo} to prove these upper bounds. Our main estimate is contained in the following lemma.

\begin{lemma}
For $\delta> 0$ and $t_{0}\in(\delta,T]$, define a quantity $\cq_{1,\delta,p}$ by
\begin{equation}\label{c81}
\cq_{1,\delta,p}
=
\sup_{y_0 \in \mathbb{R}^d}
\E \left|\sup_{t \in [t_0-\delta,t_0]} \sup_{x \in \mathbb{R}^d} \frac{|D_{h_{t_0,x_0,\delta}}u(t,x)|}{1 + |x-y_0|^\theta} \right|^p.
\end{equation}
Then under the conditions of Theorem~\ref{thm:dens-u}, there exist a constant $C>0$, depending on $p$ and $\theta$, but not $\delta$, and a parameter $\delta_0>0$ such that for all $\delta \in (0, \delta_0)$,
\begin{equation} \label{eq:Dh-u-moment-upper-bound}
    \cq_{1,\delta,p} \leq  C(Q_\Lambda(\delta))^p.
  \end{equation}
\end{lemma}

\begin{proof}
First we recall our upper bounds on the weighted supremum norms of the stochastic integral $B_{h_{t_0,x_0,\delta}}(t,x)$, where we recall that
$B_{h_{t_0,x_0,\delta}}(t,x)$ is defined by \eqref{eq:B_h-delta-interval}. Notice that
\begin{equation*}
B_{h_{t_0,x_0,\delta}}(t,x)=\ci^{\vp}(t,x),
\quad\text{with}\quad
\vp(s,y)=\sigma'(u(s,y))D_{h_{t_0,x_0,\delta}}u(s,y) \, ,
\end{equation*}
where the notation $\ci^{\vp}$ comes from \eqref{eq:I-integral-def}. If both $p>1$ and $\theta>0$ are sufficiently large, then inequality~\eqref{eq:weighted-Lp-Hilbert}  guarantees that there exist $C_{\delta,p,\theta}$ such that $\lim_{\delta \to 0} C_{\delta,p,\theta}=0$ and such that
\begin{equation} \label{eq:B_h-uppper-bound}
  \sup_{y_0 \in \mathbb{R}^d} \E \left|\sup_{t \in [t_0-\delta,t_0]} \sup_{x \in \mathbb{R}^d}
  \frac{|B_{h_{t_0,x_0,\delta}}(t,x)|}{1 + |x-y_0|^\theta} \right|^p \leq C_{\delta,p,\theta} \, \cq_{1,\delta,p} \, ,
\end{equation}
where we have just defined $\cq_{1,\delta,p}$ in \eqref{c81}.
We know that the right-hand-side of \eqref{eq:B_h-uppper-bound} is finite because of Theorem \ref{thm:u-mall-diff}.
In addition,
because $D_{h_{t_0,x_0,\delta}}u = \mathcal{L} ( \langle \phi , \, h_{t_0,x_0,\delta} \rangle_{\H_T}  + B_{h_{t_0,x_0,\delta}})$,  Theorem \ref{thm:M-Lip-inhomo} guarantees that there exists $C>0$ such that for all $\delta \in [0,1]$,
\begin{multline}\label{c9}
  \sup_{y_0 \in \mathbb{R}^d} \E \left|\sup_{t \in [t_0-\delta,t_0]} \sup_{x \in \mathbb{R}^d} \frac{|D_{h_{t_0,x_0,\delta}}u(t,x)|}{1 + |x-y_0|^\theta} \right|^p
  \leq  C\sup_{y_0 \in \mathbb{R}^d} \E \left|\sup_{t \in [t_0-\delta,t_0]} \sup_{x \in \mathbb{R}^d} \frac{|\left<\phi_{t,x}, h_{t_0,x_0,\delta}\right>_{\H_T} |}{1 + |x-y_0|^\theta} \right|^p  \\
+ C\sup_{y_0 \in \mathbb{R}^d} \E \left|\sup_{t \in [t_0-\delta,t_0]} \sup_{x \in \mathbb{R}^d} \frac{|B_{h_{t_0,x_0,\delta}}(t,x)|}{1 + |x-y_0|^\theta} \right|^p.
\end{multline}
Then plugging \eqref{eq:phi-h-product-upper} and \eqref{eq:B_h-uppper-bound} into the right hand side of~\eqref{c9}, we get
\begin{equation}\label{c10}
\sup_{y_0 \in \mathbb{R}^d} \E \left|\sup_{t \in [t_0-\delta,t_0]}
\sup_{x \in \mathbb{R}^d} \frac{|D_{h_{t_0,x_0,\delta}}u(t,x)|}{1 + |x-y_0|^\theta}
\right|^p
\le
C_{\delta, p,\theta} \cq_{1,\delta,p} + C(Q_\Lambda(\delta))^p  \lp 1+  \cq_{2,\delta,p} \rp  \, ,
\end{equation}
where we have set
\begin{equation}\label{c11}
\cq_{2,\delta,p}
=
\sup_{y_0 \in \mathbb{R}^d} \E \left|\sup_{t \in [t_0-\delta,t_0]} \sup_{x \in \mathbb{R}^d} \frac{|u(t,x)|}{1 + |x-y_0|^\theta} \right|^p  \, .\end{equation}
Observe that as an easy consequence of \eqref{eq:u_n-conv}, we have
$ \cq_{2,\delta,p}<+\infty$. Furthermore, recall that $\lim_{\delta \to 0} C_{\delta,p,\theta}=0$ in equation~\eqref{eq:B_h-uppper-bound}. Hence there exists $\delta_0>0$ small enough so that $C_{\delta,p,\theta}< \frac{1}{2}$. The left hand side of~\eqref{c10} being also of the form $\cq_{1,\delta,p}$ according to our definition~\eqref{c81}, for all $\delta \in (0,\delta_0)$ we obtain that there exists a constant $C$ such that
\begin{equation}
  \sup_{y_0 \in \mathbb{R}^d} \E \left|\sup_{t \in [t_0-\delta,t_0]} \sup_{x \in \mathbb{R}^d} \frac{|D_{h_{t_0,x_0,\delta}}u(t,x)|}{1 + |x-y_0|^\theta} \right|^p \leq  C(Q_\Lambda(\delta))^p.
\end{equation}
This proves our claim \eqref{eq:Dh-u-moment-upper-bound}.
\end{proof}

\subsection{Almost sure upper bounds on the integrals $\mathbf{A}$ and $\mathbf{B}$} \label{SS:integrals-upper}
In this section, we will prove that with probability one, $A_{h_{t_0,x_0,\delta}}(t_0,x_0)$ and $B_{h_{t_0,x_0,\delta}}(t_0,x_0)$ converge to zero much faster than $Q_\Lambda(\delta)$ as $\delta \downarrow 0$. These results, along with \eqref{eq:phi-h-product-lower}, will enable us to prove that with probability one there exists $\delta>0$ such that $D_{h_{t_0,x_0,\delta}}u(t_0,x_0)>0$.

As a first step, we show that   ${|B_{h_{t_0,x_0,\delta}}(t_0,x_0)| }/{Q_\Lambda(\delta)}$  converges to zero in probability.

\begin{lemma} \label{lem:B_h-in-prob}
  Let $(t_{0}, x_{0})\in(0,T]\times\R^{d}$, and recall that the term $B_{h_{t_0,x_0,\delta}}(t_{0}, x_{0})$ is defined by~\eqref{eq:B_h-delta-interval}.
  We work under the conditions of Theorem~\ref{thm:dens-u}. Then for any $\e>0$,
  \begin{equation}\label{d1}
    \lim_{\delta \to 0}\Pro \left(  \frac{|B_{h_{t_0,x_0,\delta}}(t_0,x_0)| }{Q_\Lambda(\delta)} >\e \right)=0 \, .
  \end{equation}
\end{lemma}

\begin{proof}
We start by applying  Chebyshev inequality in order to get
\begin{equation*}
\Pro \left(  \frac{|B_{h_{t_0,x_0,\delta}}(t_0,x_0)| }{Q_\Lambda(\delta)} >\e \right)
\leq
\frac{\E|B_{h_{t_0,x_0,\delta}}(t_0,x_0)|^p}{\e^p(Q_\Lambda(\delta))^p} \, .
\end{equation*}
 Next applying successively \eqref{eq:B_h-uppper-bound} and \eqref{eq:Dh-u-moment-upper-bound} we easily get that
 \begin{equation*}
\Pro \left(  \frac{|B_{h_{t_0,x_0,\delta}}(t_0,x_0)| }{Q_\Lambda(\delta)} >\e \right)
\leq
\frac{C_{\delta,p,\theta}}{\e^p(Q_\Lambda(\delta))^p} \cq_{1,\delta,p}
\leq \frac{C}{\e^p} C_{\delta,p,\theta} \, ,
\end{equation*}
where $C_{\delta,p,\theta}$ satisfies $\lim_{\delta\to 0}C_{\delta,p,\theta} = 0$. We have thus achieved~\eqref{d1}, which finishes the proof.
\end{proof}

Let us state a corollary to Lemma \ref{lem:B_h-in-prob} which will be important in our arguments towards positivity of the Malliavin derivative. Its proof derives from standard tools in probability theory and is omitted for sake of conciseness.

\begin{corollary}\label{cor:B-almost-sure}
Let the assumptions of Lemma \ref{lem:B_h-in-prob} prevail. Then for any sequence $\delta_k \downarrow 0$, there exists a subsequence $\delta_{k_i}$ on which ${|B_{h_{t_0,x_0,\delta_{k_i}}}(t_0,x_0)| }/{Q_\Lambda(\delta_{k_i})}$ converges to zero almost surely.
\end{corollary}

The Lebesgue integral $A_{h_{t_0,x_0,\delta}}$ is more difficult to  analyze because of the presence of the $f'(u(t,x))$ term in \eqref{eq:A_h-delta-interval}. Assumption \ref{assum:f} limits the growth to $|f'(u)| \leq K \exp(K|u|^\nu)$, but because we only have moment estimates on $\sup_{t \in [0,T]}\sup_{x \in \mathbb{R}^d} \frac{|u(t,x)|}{1 + |x-x_0|^\theta}$, there is no reason to expect to have moment estimates on $\E|f'(u(t,x))|^p$. On the other hand, the bounds on $f'(u)$ guarantee that the integral defined in \eqref{eq:A_h-delta-interval} is convergent with probability one as long as $\theta \nu < 2$. Before controlling the size of $A_{h_{t_0,x_0,\delta}}$, let us thus state
an estimate on $D_{h_{t_0,x_0,\delta}}u(t,x)$ that holds with probability one.

\begin{lemma} \label{lem:Dhu-as-conv}
For $\delta>0$ and $(t_{0}, x_{0})\in(0,T]\times\R^{d}$, let $h_{t_0,x_0,\delta}$ be defined by~\eqref{eq:h-delta-def-1}. We suppose that  the assumptions of Theorem~\ref{thm:dens-u} hold.
Consider the sequence $\{\delta_{k}; \, k\ge 1\}$ defined by $\delta_{k}=2^{-k}$. Then we have
  \begin{equation}\label{d2}
    \Pro \left( \lim_{k \to \infty} \sup_{t \in [t_0-\delta_k,t_0]} \sup_{x \in \mathbb{R}^d} \frac{\delta_k^{\frac{1}{2}}
    |D_{h_{t_0,x_0,\delta_k}} u(t,x)|}{\lp 1 + |x-x_0|^\theta\rp Q_\Lambda(\delta_k)} = 0  \right) =1.
  \end{equation}
\end{lemma}

\begin{remark}
Notice that there is an extra $\delta_k^{\frac{1}{2}}$ in the numerator of the expression in~\eqref{d2}. This extra factor will help all of this converge to zero.
\end{remark}

\begin{proof}[Proof of Lemma \ref{lem:Dhu-as-conv}]
  Recall that ${\delta}_k = 2^{-k}$. Like in the proof of Lemma~\ref{lem:B_h-in-prob}, we first apply  Chebyshev's inequality. Recalling our notation~\eqref{c81} for  $\cq_{1,\delta,p}$, we get
  \begin{equation*}
\Pro \left( \sup_{t \in [t_0-\delta_k,t_0]} \sup_{x \in \mathbb{R}^d}
\frac{ \delta_k^{\frac{1}{2}}|D_{h_{t_0,x_0,\delta_k}} u(t,x)|}{(1 + |x-x_0|^\theta )Q_\Lambda(\delta_k)}
>  \delta_k^{\frac{1}{4}}\right)
\le
\frac{\delta_{k}^{p/2} \cq_{1,\delta,p}}{Q_\Lambda(\delta_k)^{p} \delta_{k}^{p/4}} \, .
\end{equation*}
Therefore due to \eqref{eq:Dh-u-moment-upper-bound} we have
  \begin{equation}
    \Pro \left( \sup_{t \in [t_0-\delta_k,t_0]} \sup_{x \in \mathbb{R}^d} \frac{ \delta_k^{\frac{1}{2}}|D_{h_{t_0,x_0,\delta_k}} u(t,x)|}{(1 + |x-x_0|^\theta )Q_\Lambda(\delta_k)} >  \delta_k^{\frac{1}{4}}\right)
    \leq  C \delta_k^{\frac{p}{4}} \leq 2^{-\frac{pk}{4}}.
  \end{equation}
  By the Borel-Cantelli Lemma, with probability one there exists $K(\omega)$ such that for all $k\geq K(\omega)$
  \begin{equation}
     \sup_{t \in [t_0-\delta_k,t_0]} \sup_{x \in \mathbb{R}^d} \frac{ \delta_k^{\frac{1}{2}}|D_{h_{t_0,x_0,\delta_k}} u(t,x)|}{1 + |x-x_0|^\theta Q_\Lambda(\delta_k)} \leq  \delta_k^{\frac{1}{4}},
  \end{equation}
  implying  that
  \begin{equation}
    \Pro \left( \lim_{k \to \infty} \sup_{t \in [t_0-\delta_k,t_0]} \sup_{x \in \mathbb{R}^d} \frac{\delta_k^{\frac{1}{2}}|D_{h_{t_0,x_0,\delta_k}} u(t,x)|}{1 + |x-x_0|^\theta Q_\Lambda(\delta_k)} = 0  \right) =1.
  \end{equation}
This achieves the proof of \eqref{d2}.
\end{proof}

With this intermediate result on the behavior of $Du(t,x)$, we can now state a lemma estimating the integral term $A$.

\begin{lemma} \label{lem:A_h-a.s.}
Recall that we have set $\delta_{k}=2^{-k}$ for $k\ge 1$. The integral $A_{h_{t_0,x_0,\delta_k}}(t_0,x_0)$ is defined by~\eqref{eq:A_h-delta-interval}, that is
\begin{equation}\label{d3}
  A_{h_{t_0,x_0,\delta_k}}(t_0,x_0) = \int_{t_0-\delta}^{t_0} \int_{\mathbb{R}^d} G(t_0-s,x_0-y) f'(u(s,y)) D_{h_{t_0,x_0,\delta_k}}u(s,y)dyds.
\end{equation}
Our assumptions are those of Theorem~\ref{thm:dens-u}.
  Then we have
  \begin{equation}\label{d31}
    \Pro \left(\lim_{k \to \infty} \frac{|A_{h_{t_0,x_0,\delta_k}}(t_0,x_0)|}{Q_\Lambda(\delta_k)} = 0 \right)=1.
  \end{equation}
\end{lemma}
\begin{proof}
We have seen that the random variable $\cq_{2,\delta,p}$ defined by~\eqref{c11} is such that $\E[\cq_{2,\delta,p}]<\infty$, as an easy consequence of~\eqref{eq:u_n-conv}. Hence with probability one there exists a (random) constant $M(\omega)$ such that
\begin{equation}\label{d4}
  \sup_{t \in [0,t_0]} \sup_{x \in \mathbb{R}^d} \frac{|u(t,x)|}{1 + |x-x_0|^\theta} \leq M(\omega).
\end{equation}
Furthermore, by possibly increasing the value of $M(\omega)$,   Lemma \ref{lem:Dhu-as-conv} guarantees that choosing $\delta_{k}=2^{-k}$, with probability one we have
\begin{equation}\label{d5}
  \sup_{t \in [t_0-\delta_k,t_0]} \sup_{x \in \mathbb{R}^d} \frac{|D_{h_{t_0,x_0,\delta_k}} u(t,x)|}{1 + |x-x_0|^\theta } \leq  M(\omega) \frac{Q_\Lambda(\delta_k)}{\delta_k^{\frac{1}{2}}} \text{ for all } k \in \mathbb{N}.
\end{equation}
In this proof, we allow the value of $M(\omega)$ to change from line to line as long as its value remains independent of $\delta_k$.

Let us turn to the estimate on $A$ in \eqref{d3}. Owing to \eqref{eq:f-exp-bound}, \eqref{d4} and \eqref{d5}, we easily obtain
\begin{multline*}
|A_{h_{t_0,x_0,\delta_k}}(t_0,x_0)|
  \leq
  \int_{t_0-\delta_k}^{t_0} \int_{\mathbb{R}^d} G(t_0-s, x_0-y)
  \exp\left( M(\omega) |x_0-y|^{\theta \nu} \right) \\
\times  M(\omega) (1 + |y-x_0|^\theta) \frac{Q_\Lambda(\delta_k)}{\delta_k^{\frac{1}{2}}}dyds.
\end{multline*}
By increasing the value of $M(\omega)$, the integrand in the above expression is bounded by
\begin{equation}\label{d6}
  (2\pi (t_0-s))^{-\frac{d}{2}} M(\omega)\exp \left( - \frac{|y-x_0|^2}{2(t_0-s)} +  M(\omega) |x_0-y|^{\theta \nu} \right) Q_\Lambda(\delta_k)\delta_k^{-\frac{1}{2}}.
\end{equation}
We now bound the term
\begin{equation}\label{d7}
\crr_{\om,\nu,\te}
\equiv
-\frac{|y-x_0|^2}{2(t_0-s)} +  M(\omega) |x_0-y|^{\theta \nu},
\end{equation}
which appears in the exponent in \eqref{d6}. Namely recast this exponent as
\begin{equation*}
\crr_{\om,\nu,\te}
=
- \frac{|y-x_0|^2}{2(t_0-s)} + K M(\omega)(2(t_0-s))^{\frac{\theta \nu}{2}} \left(\frac{|x_0-y|}{\sqrt{2(t_0-s)}}\right)^{\theta \nu} \, .
\end{equation*}
Next because $\theta \nu<2$, Young's inequality with powers $\frac{2}{\nu \theta}$ and $\frac{2}{2-\nu \theta}$ in the right hand side above  proves that $\crr_{\om,\nu,\te}$ satisfies %\hb{(Did we already introduce $\ga$ below?)}
\begin{equation*}
\crr_{\om,\nu,\te}
\le
- \frac{|y-x_0|^2}{2(t_0-s)}  + \frac{\theta \nu|y-x_0|^2}{4(t_0-s)}
+ \frac{(2-\nu\theta)(M(\omega))^{\frac{2}{2-\nu \theta}} (2(t_0-s))^{\frac{\theta \nu}{2 - \nu \theta}} }{2}.
\end{equation*}
In addition, since $(t_0-s) < \delta_k \leq 1$, the third term in the above expression is uniformly bounded with respect to $s$ and independent of $y$. Therefore $|A_{h_{t_0,x_0,\delta_k}}(t_0,x_0)|$ is bounded by
\begin{multline*}
  |A_{h_{t_0,x_0,\delta_k}}(t_0,x_0)| \\
  \leq M(\omega)Q_{\Lambda}(\delta_k)\delta_k^{-\frac{1}{2}} \int_{t_0-\delta_k}^{t_0} \int_{\mathbb{R}^d} (2 \pi (t_0-s))^{-\frac{d}{2}} \exp \left(- \left(1- \frac{\theta\nu}{2}\right)\frac{|x_0-y|^2}{2(t_0-s)} \right)dyds.
\end{multline*}
The above integral over $\mathbb{R}^d$ is the integral of a Gaussian density and its value does not depend on $s$. Therefore, owing to the fact that the time interval is length $\delta_k$, we end up with
\begin{equation}
  |A_{h_{t_0,x_0,\delta_k}}(t_0,x_0)| \leq M(\omega) Q_{\Lambda}(\delta_k) \delta_k^{\frac{1}{2}} \, ,
\end{equation}
with probability one.
This proves that
\begin{equation}
  \lim_{k \to \infty}\frac{|A_{h_{t_0,x_0,\delta_k}}(t_0,x_0)|}{Q_\Lambda(\delta_k)} = 0 \, , \quad \text{ with probability one.}
\end{equation}
Our claim \eqref{d31} is proved.
\end{proof}

Now we can establish the main result of this subsection, which is that there exists a subsequence $\delta_k \downarrow 0$ such that with probability one,  for small values of $\delta_k$, $|A_{h_{t_0,x_0,\delta_k}}(t_0,x_0)|$  and  $|B_{h{\delta_k,t_0,x_0}}(t_0,x_0)|$ are much smaller than $Q_{\Lambda}(\delta_k)$.
\begin{proposition} \label{thm:A-B-conv-as}
Consider $(t_{0},x_{0})\in (0,T]\times\R^{d}$ and the sequence $\{\delta_{k}=2^{-k}; \, k\ge 1\}$. The function $h_{t_0,x_0,\delta_{k}}$ is introduced in~\eqref{eq:h-delta-def}, and the terms $A,B$ are respectively given by \eqref{eq:A_h-delta-interval}-\eqref{eq:B_h-delta-interval}. We assume that the coefficients $b,\si$ satisfy Assumptions~\ref{assum:sig}--\ref{assum:f}.
  Then there exists a subsequence of $\delta_{k}$, still denoted $\delta_k$, such that
  \begin{equation}
     \Pro \left(\lim_{k \to \infty} \frac{ |A_{h_{t_0,x_0,\delta_k}}(t_0,x_0)| + |B_{h_{t_0,x_0,\delta_k}}(t_0,x_0)|}{Q_\Lambda(\delta_k)} = 0 \right) =1.
  \end{equation}
\end{proposition}

\begin{proof}
  Let $\delta_k$ be a subsequence from Lemma \ref{lem:A_h-a.s.} along which
  \begin{equation}
    \Pro \left(\lim_{k \to \infty} \frac{ |A_{h_{t_0,x_0,\delta_k}}(t_0,x_0)| }{Q_\Lambda(\delta_k)} = 0 \right) =1.
  \end{equation}
  Then thanks to Corollary~\ref{cor:B-almost-sure},  there is a subsequence of $\delta_k$ (relabeled as $\delta_k$) along which
  $\lim_{k\to\infty}Q_\Lambda(\delta_k)^{-1}|B_{h_{t_0,x_0,\delta_k}}(t_0,x_0)|=0$ with probability one. This proves our result.
\end{proof}

\subsection{Positivity of the Malliavin derivative} \label{SS:Du-lower}
In this section we prove that for any $t_0\in (0,T]$ and $x_0 \in \mathbb{R}^d$  we have $\|Du(t_0,x_0)\|_{\H_T}>0$ almost surely.
This will allow to establish the existence of a density for the random variable $u(t_{0},x_{0})$.

\begin{theorem} \label{thm:positivity}
  Let $t_0\in (0,T]$ and $x_0 \in \mathbb{R}^d$. We assume that the coefficients $b,\si$ satisfy Assumptions~\ref{assum:sig}--\ref{assum:f} and we consider the solution $u$ to equation~\eqref{eq:mild-solution}. Then the following holds true:
  \begin{equation*}
    \Pro(\|Du(t_0,x_0)\|_{\H_T}>0)=1.
  \end{equation*}
\end{theorem}

\begin{proof}
  Let $h_{t_0,x_0,\delta}$ be defined by \eqref{eq:h-delta-def}. In \eqref{eq:Dhu-decomposed} we decomposed
  \begin{equation*}
    D_{h_{t_0,x_0,\delta}}u(t_0,x_0) = \left<\phi_{t_0,x_0},h_{t_0,x_0,\delta}\right>_{\H_T} + A_{h_{t_0,x_0,\delta}}(t_0,x_0)  + B_{h_{t_0,x_0,\delta}}(t_0,x_0).
  \end{equation*}
  Moreover, in \eqref{eq:phi-h-product-lower}, we proved that with probability one,
  \begin{equation*}
    \left<\phi_{t_0,x_0},h_{t_0,x_0,\delta}\right>_{\H_T} \geq \alpha Q_{\Lambda}(\delta).
  \end{equation*}
  By Theorem \ref{thm:A-B-conv-as} there exists a subsequence $\delta_k \downarrow 0$ such that
  \begin{equation*}
    \lim_{k \to \infty} \frac{ |A_{h_{t_0,x_0,\delta_k}}(t_0,x_0)| + |B_{h_{t_0,x_0,\delta_k}}(t_0,x_0)|}{Q_\Lambda(\delta_k)} = 0.
  \end{equation*}
  Therefore, with probability one, there exists a (random) $k(\omega)$ such that
  \begin{equation*}
    D_{h_{\delta_{k(\omega)},t_0,x_0}}u(t_0,x_0)>0.
  \end{equation*}
  This implies that $\|Du(t_0,x_0)\|_{\H_T}>0$ with probability one, because at least one of the directional derivatives is nonzero.
\end{proof}

We conclude this paper by proving our main result for the density of $u(t_0,x_0)$.

\begin{proof}[Proof of Theorem \ref{thm:dens-u}]
According to Proposition~\ref{prop:malliavin>0}, we have to check that the random variable $F=u(t_0,x_0)$ belongs to the Malliavin-Sobolev space $\D^{1,p}$, and that $\Pro(\|Du(t_0,x_0)\|_{\H_T}>0)=1$.
Now the fact that $u(t_0,x_0)\in\D^{1,p}$ is established in Theorem~\ref{thm:u-mall-diff}, while the condition $\Pro(\|Du(t_0,x_0)\|_{\H_T}>0)=1$ is the contents of Theorem~\ref{thm:positivity}. This implies that the law of $u(t_0,x_0)$ is absolutely continuous with Lebesgue measure.
\end{proof}

\section*{Acknowledgements}
The first author was partially supported by the Simon's Foundation Award 962543. 
The second author was partially supported by the NSF grant  DMS-1952966.

\end{document}